\documentclass{amsart}
\usepackage{amsmath,amsfonts,amssymb,amscd,amsthm,amsbsy,upref}
\usepackage{color}

\usepackage{comment}
\usepackage{amsrefs}
\usepackage{mathrsfs}

\usepackage{graphics}

\newtheorem{thm}{Theorem}[section]
\newtheorem*{thm*}{Theorem}
\newtheorem{lem}[thm]{Lemma}
\newtheorem{prop}[thm]{Proposition}
\newtheorem{cor}[thm]{Corollary}

\newtheorem{notation}[thm]{Notation}
\newtheorem{prob}[thm]{Problem}
\theoremstyle{definition}
\newtheorem{defin}[thm]{Definition}
\newtheorem{remark}[thm]{Remark}

\newcommand{\N}{\mathbb N}
\newcommand{\R}{\mathbb R}

\newcommand{\E}{\mathbb E}
\renewcommand{\P}{\mathbb P}

\newcommand{\cD}{\mathcal{D}}

\newcommand{\cE}{\mathcal{E}}

\newcommand{\vp}{\varepsilon}

\newcommand{\supp}{{\rm{supp}}}
\newcommand{\dist}{{\rm{dist}}}
\newcommand{\cof}{{\rm{cof}}}

\newcommand{\ie}{{\it i.e.}}

\newcommand{\keq}{\!=\!}

\newcommand{\kge}{\!\ge\!}
\newcommand{\kle}{\!<\!}

\newcommand{\kin}{\!\in\!}

\newcommand{\eb}{\bar e}

\newcommand{\spa}{\text{\rm span}}

\def\trin{|\!|\!|}

\usepackage[svgnames,x11names]{xcolor}

\usepackage{varioref}
\usepackage[nameinlink]{cleveref}

\usepackage[section]{placeins}
\usepackage{float}

\usepackage{xargs}

\usepackage[colorinlistoftodos,prependcaption,textsize=tiny]{todonotes}

\newcommandx{\todoin}[2][1=]{%
  \todo[inline, caption={todo}, #1]{%
    \begin{minipage}{\textwidth-20pt}#2\end{minipage}}}

\usepackage[all]{xy}

\usepackage{enumerate}

\usepackage{xifthen}

\usepackage{ifthenx}

\usepackage{xargs}

\newcounter{proof}
\newenvironment{myproof}%
{\stepcounter{proof}\begin{proof}}%
  {\end{proof}}%

\newcounter{proofstep}[proof]
{\refstepcounter{proofstep}\smallskip\par\noindent%
  \ifthenelse{\isempty{#1}}
  {\textsc{Step \theproofstep.}}
  {\textsc{#1.}}
  \noindent}%
{\par}%

\newcounter{proofcase}[proof]
\newenvironment{proofcase}[1][]%
{\refstepcounter{proofcase}\smallskip\par\noindent%
  \ifthenelse{\isempty{#1}}
  {\textsc{Case \theproofcase.}}
  {\textsc{#1.}}
  \noindent}%
{\par}%

\DeclareMathOperator{\lesslex}{<_\ell}

\DeclareMathOperator{\drless}{\lhd\,}
\DeclareMathOperator{\drlesseq}{\unlhd\,}
\newcommand{\drindex}{\mathcal{O}_{\drless}}

\DeclareMathOperator{\cond}{\mathbb{E}}

\allowdisplaybreaks

\begin{document}

\title[Strategical reproducibility and factorization]{Strategically reproducible bases and the
  factorization property}

\author{R. Lechner}

\address{R. Lechner, Institute of Analysis, Johannes Kepler University Linz, Altenberger Strasse 69,
  A-4040 Linz, Austria}

\email{Richard.Lechner@jku.at}

\author{P. Motakis}

\address{P. Motakis, Department of Mathematics, University of Illinois at Urbana-Champaign, Urbana,
  IL 61801, U.S.A.}

\email{pmotakis@illinois.edu}

\author{P.F.X. M\"uller}

\address{P.F.X. M\"uller, Institute of Analysis, Johannes Kepler University Linz, Altenberger
  Strasse 69, A-4040 Linz, Austria}

\email{Paul.Mueller@jku.at}

\author{Th.~Schlumprecht}

\address{Th.~Schlumprecht, Department of Mathematics, Texas A\&M University, College Station, TX
  77843-3368, USA, and Faculty of Electrical Engineering, Czech Technical University in Prague,
  Zikova 4, 16627, Prague, Czech Republic}

\email{schlump@math.tamu.edu}

\thanks{The first and the third author was supported by the Austrian Science Foundation (FWF) under
  Grant Number Pr.Nr. P28352.  The second named author was supported by the National Science
  Foundation under Grant Number DMS-1600600.  The fourth named author was supported by the National
  Science Foundation under Grant Numbers DMS-1464713 and DMS-1711076 .}

\keywords{}

\subjclass[2010]{46B25, 47A68, 30H10}

\begin{abstract}
  We introduce the concept of strategically reproducible bases in Banach spaces and show that
  operators which have large diagonal with respect to strategically reproducible bases are factors
  of the identity.  We give several examples of classical Banach spaces in which the Haar system is
  strategically reproducible: multi-parameter Lebesgue spaces, mixed-norm Hardy spaces and most
  significantly the space $L^1$.  Moreover, we show the strategical reproducibility is inherited by
  unconditional sums.
\end{abstract}

\maketitle


\makeatletter

\providecommand\@dotsep{5}

\def\listtodoname{List of Todos}

\def\listoftodos{\@starttoc{tdo}\listtodoname}

\makeatother


\section{Introduction}\label{S:0}

In this paper, we address the following question: Given a Banach space $X$ with a basis
$(e_i)_{i=1}^\infty$, let $T : X\to X$ be an operator, whose matrix representation has a diagonal
whose elements are uniformly bounded away from $0$. We say in that case $T$ has a \emph{large
  diagonal}. Is it possible to factor the identity operator on $X$ through $T$?

The origin of this problem can be traced back to the work of
Pe{\l}czy{\'n}ski~\cite{pelczynski:1960}, who proved that every infinite dimensional subspace of
$\ell^p$, $1 \leq p < \infty$ and $c_0$ contains a further subspace which is complemented and
isomorphic to the whole space.

Closely related is the concept of primarity of a Banach space.  Recall that $X$ is called
\emph{primary}, if for every bounded projection $P : X\to X$, either $P(X)$ or $(I - P)(X)$ is
isomorphic to $X$.  The connection between the primarity of a Banach space and the factorization
problem is as follows: either $P$ has large diagonal or $I-P$ has large diagonal on a ``large''
subsequence of the basis $(e_i)_{i=1}^\infty$ of the Banach space $X$.  For example Maurey
\cite{maurey:sous:1975} proved primarity for $X=L^p$, $1\leq p < \infty$, by showing that for every
operator $T : L^p\to L^p$, the identity operator factors either through $T$ or $I-T$; see also
Alspach-Enflo-Odell~\cite{alspach:enflo:odell:1977}.  Factorization and primarity theorems were
obtained by Capon \cite{capon:1982:2} for the mixed norm spaces $L^p(L^q)$, $1 < p,q < \infty$, and
by the third named author \cite{mueller:1988} for $H^1$ and $\mathrm{BMO}$.

Separately, Andrew~\cite{Andrew1979} showed that for $1 < p < \infty$, every operator
$T : L^p\to L^p$ which has large diagonal with respect to the Haar system is a factor of the
identity operator on $L^p$.  More recently in~\cite{laustsen:lechner:mueller:2015} it was proved
that for $1 \leq p,q < \infty$, every operator $T : H^p(H^q)\to H^p(H^q)$ which has large diagonal
is a factor of the identity operator on $H^p(H^q)$.

In this paper we introduce a new approach to the factorization problem by devising an infinite two
person game and isolate a property of a basis called \emph{strategical reproducibility}, which
implies the factorization through the identity of operators with large diagonal.  We say in that
case, the basis $(e_i)_{i=1}^\infty$ has the \emph{factorization property}.  By using this method,
we obtain simplified proofs of existing results, and obtain the following new factorization theorems
for $L^1$ and related spaces.

\begin{thm*}
  The normalized Haar system of $L^1[0,1]$ has the factorization property.  Moreover, the normalized
  bi-parameter Haar system of $L^1([0,1]^2)$ and the tensor product of the $\ell^p$ unit vector
  basis with the Haar system have the factorization property.
\end{thm*}

The paper is organized as follows.  Section~\ref{sec:basic-concepts} covers basic concepts relevant
to this work.  In Section~\ref{sec:strat-repr-cond}, we define three notions of strategical
reproducibility and show that those imply the factorization property.  In
Section~\ref{sec:basic-overview-multi} we review basic properties of multi-parameter Lebesgue- and
Hardy spaces.  In Section~\ref{sec:strat-repr-haar} we establish that the Haar system is
strategically reproducible in several classical Banach spaces such as reflexive, multi-parameter
Lebesgue spaces, $H^1$ and two-parameter Hardy spaces $H^p(H^q)$, $1\leq p,q < \infty$.  In
Section~\ref{sec:haar-system-l1_0-1} we show that the Haar system is strategically reproducible in
$L^1_0$.  In Section~\ref{sec:uncond-sums-spac} we show that unconditional sums of spaces with
strategically reproducible bases have themselves that property.  Finally, we discuss open problems
in Section~\ref{sec:final-comments-open}.

\section{A brief discussion of basic concepts}
\label{sec:basic-concepts}

We discuss several closely related concepts for operators on Banach spaces.
 
\begin{defin}
  Let $X$ be a Banach space and $T:X\to X$ be a bounded linear operator.
  \begin{itemize}
  \item[(i)] We say that $T(X)$ {\em contains a copy of $X$} if there is a (necessarily closed)
    subspace $Y$ of $T(X)$ that is isomorphic to $X$.
  \item[(ii)] We say that $T$ {\em preserves a copy of $X$ (or fixes a copy of $X$)} if there exists
    a subspace $Y$ of $X$ that is isomorphic to $X$ and $T$ restricted on $Y$ is an into
    isomorphism.
  \item[(iii)] We say that { \em the identity operator $I$ on $X$ factors through $T$} if there are
    bounded linear operators $R,S:X\to X$ with $I = STR$.
  \end{itemize}
  We also consider a quantified version of (iii). For $K>0$ we say that {\em the identity
    $K$-factors through $T$ } if there are bounded linear operators $R,S:X\to X$ with
  $\|R\|\cdot\|S\|\le K$ and $I = STR$ and we say that {\em the identity almost $K$-factors through
    $T$ } if it $(K+\vp)$-factors through $T$ for all $\vp>0$.
\end{defin}

\begin{remark}
  In general, for a given operator $T$, it is easy to see that (iii)$\Rightarrow$(ii) and
  (ii)$\Rightarrow$(i). The converse implications are in general false. To see that
  (i)$\not\Rightarrow$(ii) take a quotient operator $T_0:L^1\to\ell_1$ and a quotient operator
  $T_1:\ell_1\to L^1$. Then if $T = T_1\circ T_0:L^1\to L^1$, $T(L^1) = L^1$ however $T$ does not
  preserve a copy of $L^1$. There is an example demonstrating (iii)$\not\Rightarrow$(ii) but it is
  slightly more involved. We first observe that if $I = STR$ and $Z = TR(X)$, then $Z$ is isomorphic
  to $X$ and complemented in $X$. Indeed, it follows that $R$ is bounded below and $T$ is bounded
  below on $R(X)$ hence $TR$ is an isomorphic embedding. Furthermore, $S$ restricted on $Z = TR(X)$
  is an isomorphism onto $X$. Therefore, we can define the inverse map $S|_{Z}^{-1}:X\to Z$. One can
  check that $Px = S|_{Z}^{-1}(Sx)$ defines a bounded projection onto $Z$. This easy fact implies
  that if $X$ is a minimal space that is not complementably minimal then there exists an operator
  $T:X\to X$ that is an into isomorphism so that the identity does not factor through $X$. To see
  this, choose a subspace $Y$ of $X$ that is isomorphic to $X$ and does not contain a further
  subspace isomorphic to $X$ and complemented in $X$. If $T:X\to X$ is an into isomorphism, the
  image of which is $Y$, then the identity does not factor through $T$. Indeed, if $I = STR$ then
  $Z = TR(X)$ is isomorphic to $X$ and complemented to $X$. This is not possible because $Z$ is a
  subspace of $Y$. In conclusion, the fact that (iii)$\not\Rightarrow$(ii) is reduced to the
  existence of a minimal and not complementably minimal space $X$. It is well known that the dual of
  Tsirelson space has this property, however to the best of our knowledge there is no recorded proof
  of this fact so we give a short description of it here. Assume that $T^*$ is complementably
  minimal. We will show that this would imply that $T$ is minimal, which was proved to be false in
  \cite{CasazzaShura1989}*{Corollary VI.b.6, page 58}. Let $X$ be an infinite dimensional subspace
  of $T$. By \cite{CasazzaJohnsonTzafriri1984}*{Theorem 1} $X$ is isomorphic to a quotient of $T$
  and hence $X^*$ is isomorphic to a subspace of $T^*$. If $T^*$ is complementably minimal, then
  $X^*$ contains a complemented copy of $T^*$ which yields that $X$ contains a complemented copy of
  $T$. In particular, $T$ is minimal and this cannot be the case.
\end{remark}

The following definition of $C$-perturbable was introduced by Andrew in~\cite{Andrew1979}.  The
concept of large diagonal, which was implicitly present in~\cite{Andrew1979}, was formally
introduced in~\cite{laustsen:lechner:mueller:2015}.
\begin{defin}
  \label{factorization definitions}
  Let $X$ be a Banach space with a normalized Schauder basis $(e_k)_k$.
  \begin{itemize}
  \item[(i)] Let $0<C\leq 1$. The basis $(e_k)_k$ is called {\em $C$-perturbable} if whenever
    $T:X\to X$ is a bounded linear operator for which there exists $\delta>0$ with
    $\|T(e_k)-e_k\|<C-\delta$ for all $k\in\N$, then $T(X)$ contains a copy of $X$.
  \item[(ii)] If an operator $T$ on $X$ satisfies $\inf_k\big|e_k^*(T(e_k))\big|>0$, then we say
    that $T$ has \emph{large diagonal}.
  \item[(iii)] An operator $T$ on $X$ satisfying $e^*_m(T(e_k) )= 0$ whenever $k\neq m$, is called
    \emph{diagonal operator}.
  \item[(iv)] We say that the basis $(e_k)_k$ has the {\em factorization property} if whenever
    $T:X\to X$ is a bounded linear operator with $\inf_k|e_k^*(Te_k)|>0$ then the identity of $X$
    factors through $T$.
  \end{itemize}
\end{defin}

\begin{remark}
  It is easy to see that a basis that has the factorizing property is also 1-perturbable. However,
  there are bases that are $C$-perturbable without the factorization property, as the following
  example shows.
\end{remark}

The norm on the boundedly complete basis of James space $(e_i)_i$ is defined as follows:
\begin{equation}
  \label{james formula}
  \Big\|\sum_{i=1}^na_ie_i\Big\| = \sup\Big(\sum_{k=1}^m\Big(\sum_{i\in E_k}a_i\Big)^2\Big)^{1/2},
\end{equation}
where the supremum is taken over $m\in\mathbb{N}$ and sequences of successive intervals
$(E_k)_{k=1}^m$ of natural numbers. Let $J$ denote the completion of the linear span of $(e_i)_i$
with this norm. Some well known important properties of $J$ are the following:
\begin{itemize}

\item[(i)] The basis $(e_i)_i$ is spreading. In particular for any sequence scalars $(a_i)_{i=1}^n$
  and natural numbers $k_1<\cdots<k_n$ we have
  \begin{equation}
    \label{jamespreading}
    \Big\|\sum_{i=1}^na_ie_i\Big\| = \Big\|\sum_{i=1}^na_ie_{k_i}\Big\|.
  \end{equation}

\item[(ii)] The sequence $(e_i)_i$ is non-trivial weak Cauchy, i.e., there is
  $e^{**}\in J^{**}\setminus J$ so that $w^*\text{-}\lim_ie_i = e^{**}$. Additionally,
  $\mathrm{dist}(e^{**},J) = 1$.

\item[(iii)] The space $J$ is quasi-reflexive of order one. In particular,
  $J^{**} = \mathbb{R}e^{**}\oplus J$, i.e., $J^{**}$ is spanned by $e^{**}$ and the canonical
  embedding of $J$ in $J^{**} $.

\end{itemize}

Note that by (iii) if the identity factors through an operator on $J$ then that operator cannot be
weakly compact.

\begin{prop}
  \label{jamesnotfac}
  There exists a weakly compact operator $T:J\to J$ with $e_i^*(T(e_i) )= 1$, for all
  $i\in\mathbb{N}$. In particular, $(e_i)$ dos not have the factorization property in $J$.
\end{prop}

\begin{proof}
  By (ii) the operator $S:J\to J$ given by $Se_i = e_{i+1}$, for all $i\in\mathbb{N}$, is a linear
  isometry. We define $T = I - S$, which has norm at most two. We will show that $S$ is weakly
  compact by showing that for every bounded sequence $(x_i)_i$ the sequence $(Tx_i)_i$ has a weakly
  convergent subsequence. By the separability of $J^*$, we pass to a subsequence so that $(x_i)_i$
  converges in the $w^*$-topology to some $x^{**}\in J^{**}$. By (iii) there is $x\in J$ and
  $c\in\mathbb{R}$ so that $x^{**} = x + ce^{**}$. We have
  \begin{align*}
    S^{**}(x^{**} )
    &= S(x) + cS^{**}(e^{**} )
      = S(x) + c\Big(w^*\text{-}\lim_iS(e_i)\Big)\\
    &= S(x) + c\Big(w^*\text{-}\lim_ie_{i+1}\Big)
      = S(x) + ce^{**}.
  \end{align*}
  Thus $w^*$-$\lim_iTx_i =
  w^*$-$\lim_i \big(x_i - S(x_i)\big) = x + ce^{**} - \big(S(x)+ce^{**}\big) = x-S(x)$.  Because the
  $w^*$-limit is in $J$ it has to be a weak limit.
\end{proof}

We wish to show now that the boundedly complete basis of James space is perturbable. To achieve that
we shall need the following well known fact. We describe a proof for completeness.

\begin{prop}
  \label{many projections}
  Let $(x_i)_i$ be a non-trivial weak Cauchy sequence in $J$. Then, $(x_i)_i$ has a subsequence
  $(x_{j_i})_i$ that is equivalent to $(e_i)_i$ so that there exists a bounded linear projection
  $P:J\to[(x_{j_i})_i]$.
\end{prop}

\begin{proof}
  Proposition 7.4 from \cite{ArgyrosMotakisSari2017} says the result holds, provided that the
  sequence $(e_i)_i$ is equivalent to its convex block sequences and not equivalent to the summing
  basis of $c_0$.  Both of these properties follow from \eqref{james formula}.
\end{proof}

\begin{prop}
  \label{james pert}
  Let $T:J\to J$ be a bounded linear operator with the property $\sup_i\|Te_i - e_i\| < 1$. Then the
  identity factors through $T$. That is, the boundedly complete basis of $J$ is perturbable.
\end{prop}

\begin{proof}
  If $C = \lim\inf\|T(e_i)-e_i\| <1$ then we have $\|T^{**}(e^{**} )- e^{**}\| \leq C$ and
  therefore, from (ii) $\mathrm{dist}(T^{**}(e^{**}),J) \geq 1-C>0$. This means that
  $T^{**}(e^{**})$, which is the $w^*$-limit of $(T(e_i))_i$, is not in $J$. In other words,
  $(T(e_i))_i$ is non-trivial weak Cauchy. By Proposition \ref{many projections} there is a
  subsequence $(T(e_{j_i}))_i$ of $(T(e_i))_i$ that is equivalent to $(e_i)_i$ and a bounded linear
  projection $P:J\to W = [(T(e_{j_i}))_i]$. Let $A:J\to J$ be the map defined by $Ae_i = e_{j_i}$,
  which by (i) is bounded. Let $R:W\to J$ be the isomorphism given by $R(T(e_{j_i})) = e_i$ and set
  $B:J\to J$ with $B = R\circ P$. It is easy to see that $I = B\circ T\circ A$.
\end{proof}

\section{Strategical Reproducibility, a condition implying the factorization property}
\label{sec:strat-repr-cond}

In this section we formulate several versions of a property of bases we call \emph{strategical
  reproducibility} and show that they imply the factorization property.

\subsection*{Notation and conventions}

All our Banach spaces are assumed to be over the real numbers $\R$.  $B_X$ denotes the unit ball,
and $S_X$ denotes the unit sphere of a Banach space $X$.  $c_{00}$ denotes the sequences in $\R$
which eventually vanish.

For a Banach space $X$ we denote by $\cof(X)$ the set of cofinite dimensional subspaces of $X$,
while $\cof_{w^*}(X^*)$ denotes the cofinite dimensionl $w^*$-closed subspaces of $X^*$.

If $\eb=(e_i)$ is a basis of a Banach space $X$, we call for $x=\sum_{i=1}^\infty x_ie_i\in X$ the
set $\{i\in\N: x_i\not= 0\}$ {\em the support of $x$ with respect to $(\eb)$} and denote it by
$\supp_{\eb}(x)$. If there is no confusion possible, we also may write $\supp(x)$ instead of
$\supp_{\eb}(x)$.

We recall that a basis $(e_n)$ of a Banach space $X$ is {\em shrinking} if the coordinate
functionals $(e^*_n)$ are a basis of $X^*$, and {\em unconditional} if for some constant $c\ge 1$
and all finite sequence of scalars $(a_i)_{i=1}^n$, and all $\sigma=(\sigma_i)_{i=1}^n\in \{\pm 1\}$
\begin{equation*}
  \Big\|\sum_{i=1}^n \sigma_i a_i e_i\Big\|\le c\Big\|\sum_{i=1}^n  a_i e_i\Big\|.
\end{equation*}

Let $(x_i)_i$ and $(y_i)_i$ be Schauder basic sequences in (possibly different) Banach spaces and
$C\geq 1$. We say that $(x_i)_i$ and $(y_i)_i$ are {\em $C$-equivalent} if there are $A,B>0$, with
$A\cdot B\le C$ so that for any $(a_i)_i\in c_{00}$ of scalars we have
\begin{equation*}
  \frac1A\Big\|\sum_{i=1}^\infty a_iy_i\Big\|\le \Big\|\sum_{i=1}^\infty a_ix_i\Big\|\le B\Big\|\sum_{i=1}^\infty a_iy_i\Big\|.
\end{equation*}
We say that $(x_i)_i$ and $(y_i)_i$ are {\em ``impartially $C$-equivalent''} if for any finite
choice of scalars $(a_i)_i\in c_{00}$ we have
\begin{equation*}
  \frac{1}{\sqrt C}\Big\|\sum_{i=1}^\infty a_iy_i\Big\| \leq \Big\|\sum_{i=1}^\infty a_ix_i\Big\| \leq \sqrt{C}\Big\|\sum_{i=1}^\infty a_iy_i\Big\|.
\end{equation*}
Note that if two sequences are $C$-equivalent then by scaling one of them we can always make them
impartially $C$-equivalent.

We now formally define the concept of \emph{strategical reproducibility} depending on properties of
the basis of a Banach space.  The most general form will be given in Defitinion~\ref{D:1.3}.
Nevertheless, under additional assumptions on the basis, this notion considerably simplifies.  The
proof that the different definitions of strategical reproducibility are equivalent under their
respective assumptions on the basis will be given later.

If we demand that our basis is unconditional and shrinking strategical reproducibility can be
defined as follows.
\begin{defin}\label{D:1.1}
  Assume that $X$ is a Banach space with a basis $(e_n)$ which is unconditional and shrinking. Let
  $(e^*_n)\subset X^*$ be the corresponding coordinate functionals.  We say that $(e_n)$ is {\em
    strategically reproducible} if the following condition is satisfied for some $C\ge 1$:

  \begin{align}
    &\label{E:1.1.1}\qquad\qquad\forall n_1\kin\N\,\exists b_1\kin \spa(e_n:n\kge n_1) \,\exists b_1^*\kin \spa(e^*_n:n\ge n_1)\\
    &\qquad\qquad \forall n_2\kin\N\,\exists b_2\kin \spa(e_n:n\kge n_2)
      \,\exists b_2^*\kin \spa(e^*_n:n\kge n_2)\notag\\
    &\qquad \qquad\forall n_3\kin\N \,\exists b_3\kin \spa(e_n:n\kge n_3)
      \,\exists b_3^*\kin \spa(e^*_n:n\kge n_3) \notag\\
    &\qquad\qquad\vdots \notag \\
    &\qquad\qquad\text{so that:} \notag\\                                                  
    &\qquad\qquad(b_{k}) \text{ is impartially  $C$-equivalent to }(e_k),\text{ and }\tag{\ref{E:1.1.1}a}\\
    &\qquad\qquad(b^*_{k}) \text{ is impartially $C$-equivalent to }(e^*_k), \notag\\
    &\qquad\qquad b^*_k(b_l)=\delta_{k,l}\text{ for all $k,l\in\N$.} \tag{\ref{E:1.1.1}b}
  \end{align}
\end{defin}

\begin{remark} 
  Condition \eqref{E:1.1.1} in Definition \ref{D:1.1} can be interpreted that one player in a
  two-person game has a winning strategy:

  We fix $C\ge1$.  Player I chooses $n_1\kin\N$, then player (II) chooses
  $b_1\kin \spa(e_n:n\kge n_1)$ and $b_1^*\kin \spa(e^*_n:n\kge n_1)$.  They repeat the moves
  infinitely many times, obtaining for every $k\kin\N$ numbers $n_k$, and vectors $b_k$ and $b^*_k$.
  Player II wins if he was able to choose the sequences $(b_n)\subset X$ and $(b^*_n)\subset X^*$ so
  that (\ref{E:1.1.1}a), (\ref{E:1.1.1}b) are satisfied.  Thus, the basis $(e_i)$ is strategically
  reproducible if and only if for some $C\ge 1$ player (II) has a winning strategy.

  In general it is not true that two-player games of infinite length are determined, \ie, that one
  of the player has a winning strategy. Nevertheless, for $C\ge 1$ it is easy to see that the set of
  all sequences $(b_k,b^*_k)$ in $(X\times X^*)^{\N}$ which satisfy \ref{E:1.1.1}(a) and
  \ref{E:1.1.1}(b) is Borel measurable (it is actually closed) with respect to the product topology
  of the discrete topology on $X\times X^*$, and thus it follows from the main result in
  \cite{Martin1975} that this game is determined. More on these {\em Infinite Asymptotic Games} can
  be found in \cite{OdellSchlumprecht2002}.
\end{remark}

Now we relax the condition on our basis $(e_i)$ and only require it be unconditional. In that case
we define {\em strategical reproducible} as follows.
\begin{defin}\label{D:1.2}
  Let $X$ be a Banach space with an unconditional basis $(e_i)_i$ and fix positive constants
  $C\geq 1$.

  Consider the following two-player game between player (I) and player (II). For $k\in\N$, turn $k$
  is played out in three steps.
  \begin{itemize}
  \item[Step\! 1:] Player (I) chooses $\eta_k>0$, $W_k\in\mathrm{cof}(X)$, and
    $G_k\in\mathrm{cof}_{w^*}(X^*)$,
  \item[Step\! 2:] Player (II) chooses a finite subset $E_k$ of $\N$ and sequences of non-negative
    real numbers $(\lambda_i^k)_{i\in E_k}$, $(\mu_i^k)_{i\in E_k}$ satisfying
    \begin{equation*}
      \sum_{i\in E_k}\lambda_i^{(k)}\mu_i^{(k)} = 1.
    \end{equation*}
  \item[Step\! 3:] Player (I) chooses $(\vp_i^{(k)})_{n\in E_k}$ in $\{-1,1\}^{E_k}$.

  \end{itemize}

  We say that player (II) has a winning strategy in the game $\mathrm{Rep}_{(X,(e_i))}(C)$ if he can
  force the following properties on the result:

  For all $n\in\N$ we set
  $x_k = \sum_{i\in E_k}\vp_i^{(k)} \lambda^{(k)}_ie_i \text{ and }x_k^* = \sum_{i\in
    E_k}\vp_i^{(k)}\mu^{(k)}_ie^*_i$ and demand:
  \begin{itemize}
  \item[(i)] the sequences $(x_k)_k$ and $(e_k)_k$ are impartially $C$-equivalent,
  \item[(ii)] the sequences $(x_k^*)_k$ and $(e_k^*)_k$ are impartially $C$-equivalent,
  \item[(iii)] for all $n\in\N$ we have $\mathrm{dist}(x_k, W_k) < \eta_k$, and
  \item[(iv)] for all $n\in\N$ we have $\mathrm{dist}(x^*_k, G_k) < \eta_k$.
  \end{itemize}
  We say that $(e_i)_i$ is {\em $C$-strategically reproducible in $X$} if for every $\eta >0$ player
  II has a winning strategy in the game $\mathrm{Rep}_{(X,(e_i))}(C+\eta)$.
\end{defin}

Finally we will not even require the basis $(e_j)$ to be unconditional and define {\em strategical
  reproducible} as follows:

\begin{defin}\label{D:1.3}
  Let $X$ be a Banach space with a normalized Schauder basis $(e_i)_i$ and fix positive constants
  $C\geq 1$, and $\eta>0$.

  Consider the following two-player game between player (I) and player (II):

  Before the first turn player (I) is allowed to choose a partition of $\N = N_1\cup N_2$. For
  $k\in\N$, turn $k$ is played out in three steps.
  \begin{itemize}
  \item[Step\! 1:] Player (I) chooses $\eta_k>0$, $W_k\in\mathrm{cof}(X)$, and
    $G_k\in\mathrm{cof}_{w^*}(X^*)$,
  \item[Step\! 2:] Player (II) chooses $i_k\in\{1,2\}$, a finite subset $E_k$ of $N_{i_k}$ and
    sequences of non-negative real numbers $(\lambda_i^k)_{i\in E_k}$, $(\mu_i^k)_{i\in E_k}$
    satisfying
    \begin{equation*}
      1-\eta <\sum_{i\in E_k}\lambda_i^{(k)}\mu_i^{(k)}< 1+\eta.
    \end{equation*}
  \item[Step\! 3:] Player (I) chooses $(\vp_i^{(k)})_{n\in E_k}$ in $\{-1,1\}^{E_k}$.

  \end{itemize}

  We say that player (II) has a winning strategy in the game $\mathrm{Rep}_{(X,(e_i))}(C,\eta)$ if
  he can force the following properties on the result:

  For all $n\in\N$ we set
  $x_k = \sum_{i\in E_k}\vp_i^{(k)} \lambda^{(k)}_ie_i \text{ and }x_k^* = \sum_{i\in
    E_k}\vp_i^{(k)}\mu^{(k)}_ie^*_i$ and demand:
  \begin{itemize}
  \item[(i)] the sequences $(x_k)_k$ and $(e_k)_k$ are impartially $(C+\eta)$-equivalent,
  \item[(ii)] the sequences $(x_k^*)_k$ and $(e_k^*)_k$ are impartially $(C+\eta)$-equivalent,
  \item[(iii)] for all $n\in\N$ we have $\mathrm{dist}(x_k, W_k) < \eta_k$, and
  \item[(iv)] for all $n\in\N$ we have $\mathrm{dist}(x^*_k, G_k) < \eta_k$.
  \end{itemize}
  We say that $(e_i)_i$ is {\em $C$-strategically reproducible in $X$} if for every $\eta >0$ player
  II has a winning strategy in the game $\mathrm{Rep}_{(X,(e_i))}(C,\eta)$.
\end{defin}

\begin{remark}
  We first want to observe that if $(e_i)$ is a normalized shrinking and unconditional basis then
  being strategically reproducible in the sense of Definition \ref{D:1.1} is equivalent with being
  $C$-strategically reproducible for some $C\ge 1$ in the sense of Definition \ref{D:1.3}.

  Indeed, assume $(e_i)$ is $1$-unconditional and shrinking and assume that for some $\tilde C\ge 1$
  \eqref{E:1.1.1} of Definition \ref{D:1.1} holds. We will show that $(e_i)$ is
  $3\tilde C $-strategically reproducible in the sense of Definition \ref{D:1.3}.

  Let $1/3>\eta>0$ be given and assume player (I) has at the beginning of the game chosen a
  partition $(N_1,N_2)$ of $\N$.  At the $k$-th step player (I) chooses $\eta_k>0$ and spaces
  $W_k\in\cof(X)$ and $G_k\in\cof_{w^*}(X)$. Since $(e_k)$ is shrinking, player (II) can
  ``approximate $W_k$ by a tail space'' as follows: there is $n^{(1)}_k\in\N$ so that for all
  $x\in B_X\cap[e_i:i\kge n^{(1)}_k]$ it follows that $\dist(x,W_k)\kle \eta_k/2\tilde C$. Secondly,
  since $G_k$ is a $w^*$-closed and cofinite dimensional subspace of $X^*$, and thus the annihilator
  of a finite subset of $X$, we find $n^{(2)}_k \kin\N$ so that for all
  $x^*\kin B_{X^*}\cap[e^*_i:i\kge n^{(2)}_k]$ it follows that $\dist(x^*,W_k)<\eta_k/2\tilde C$.
  Finally we let $n^{(3)}_k=1+\max\big(\bigcup_{j=1}^{k-1} \supp(x_j)\cup\supp(x^*_j)\big)$.  Let
  $n_k=\max(n^{(1)}_k, n^{(2)}_k, n^{(3)}_k)$ and let player (II) follow his winning strategy,
  assuming player (I) has chosen $n_k\in \N$ in his $k$-th move of the game described in Definition
  \ref{D:1.1}, and let $b_k\in [e_i:i\ge n_k]$ and $b^*_k\in [e^*_i:i\ge n_k]$ be chosen according
  to that strategy, which in particular implies that $\|b_k\|,\|b_k^*\| \leq \sqrt{\tilde C}$.  We
  write $b_k$ and $b^*_k$ as
  \begin{equation*}
    b_k=\sum_{j=1} ^\infty \tilde\lambda^{(k)}_j e_j \text{ and } b^*_k=\sum_{j=1} ^\infty \tilde\mu^{(k)}_j e_j .
  \end{equation*}
  By reducing the supports, if necessary we can assume, by using Proposition \ref{P:1.3}, that
  $\tilde E_k=\supp(b_k)=\supp(b^*_k)$ and since
  $b_k^*(b_k)= \sum_{j=1}^\infty \tilde\mu^{(k)}_j\tilde\lambda^{(k)}_j=1$ we can choose
  $i_k\in\{1,2\}$, so that
  \begin{equation*}
    \rho_k= \sum_{j\in N_{i_k}}\tilde\mu^{(k)}_j\tilde\lambda^{(k)}_j\ge \frac{1}{2}.
  \end{equation*}
  Then we let $E_k\keq\tilde E_k\cap N_{i_k}$, $\mu^{(k)}_j\keq\tilde\mu^{(k)}_j/\sqrt{\rho_k} $ and
  $\lambda^{(k)}_j\keq\tilde\lambda^{(k)}_j/\sqrt{\rho_k}$ for $j\kin E_k$. After player (I) has
  chosen $(\varepsilon^{(k)}_j)_{j\in E_k}$ we also put
  $x_k=\sum_{j\in E_k} \varepsilon^{(k)}_j\mu^{(k)}_j e_j\kin \sqrt{2\tilde C}B_X\cap [e_j: j\kge
  n_k]$ and
  $x^*_k=\sum_{j\in E_k} \varepsilon^{(k)}_j\lambda^{(k)}_j e^*_j\kin \sqrt{2\tilde
    C}B_{X^*}\cap[e^*_i:i\kge n_k]$.

  From the choice of $n_k$, and the fact that $\|x_k\|\le \sqrt{2\tilde C}$ and
  $\|x^*_k\|\le \sqrt{2\tilde C}$, it follows that $\dist(x_k,W_k)<\eta_k$ and
  $\dist(x^*,G_k)<\eta_k$.  From the $1$-unconditionality of $(e_j)$ it follows for
  $(\xi_k)\in c_{00}$ that
  \begin{equation*}
    \Big\|\sum_{k=1}^\infty \xi_k x_k\Big\|\le \sqrt 2\Big\|\sum_{k=1}^\infty \xi_k b_k\Big\|\le \sqrt2 \sqrt{\tilde C}\le \Big\|\sum_{k=1}^\infty \xi_k e_k\Big\|,
  \end{equation*}
  and
  \begin{equation*}
    \Big\|\sum_{k=1}^\infty \xi_k x^*_k\Big\|\le \sqrt 2\Big\|\sum_{k=1}^\infty \xi_k b^*_k\Big\|\le \sqrt2 \sqrt{\tilde C}\le \Big\|\sum_{k=1}^\infty \xi_k e^*_k\Big\|.
  \end{equation*}
  Thus, by Proposition \ref{P:1.3} below, $(x_k)$ is impartially $2\tilde C$-equivalent to $(e_k)$
  and $(e^*_k)$ is impartially $2\tilde C$-equivalent to $(e^*_k)$.

  Conversely, it is easy to deduce that if $(e_j)$ is unconditional and shrinking and strategically
  reproducible in the sense of Definition \ref{D:1.3}, then it is also strategically reproducible in
  the sense of Definition \ref{D:1.1}.

  In a similar way we can show that for an unconditional and normalized basis $(e_j)$ {\em
    strategical reproducibility } in sense of Definitions \ref{D:1.2} and \ref{D:1.3} are
  equivalent.
\end{remark}

\begin{remark}
  The unit vector basis of $\ell_1$ has the factorization property yet it is not strategically
  reproducible under any of the above definitions. It is possible to give a fourth notion of
  strategic reproducibility that covers $\ell_1$, is strictly less restrictive than Definition
  \ref{D:1.3}, and implies the factorization property. This formulation is rather technical, so we
  will not discuss it in the present paper.
\end{remark}

We will now show that a basis which is strategically reproducible has the factorization property.
We will first need the following two observations.
\begin{lem}\label{L:1.2}
  Assume that $X$ is a Banach space with a basis $(e_n)$, whose basis constant is $\lambda\ge1$ and
  biorthogonal functionals $(e_n^*)$.  Let $(b_n)$ and $(b^*_n)$ be block bases of $(e_n)$ and
  $(e^*_n)$, respectively, so that $b_m^*(b_n)=\delta_{m,n}$, for $m,n\in\N$, and so that for some
  $C\ge 1$ it follows that
  \begin{equation}\label{eq:basis-estimates:1}
    \Big\|\sum_{j=1}^\infty \xi_j b_j\Big\|_X
    \le \sqrt{C} \Big\|\sum_{j=1}^\infty \xi_j e_j\Big\|_X
    \qquad\text{and}\qquad
    \Big\|\sum_{j=1}^\infty \xi_j b^*_j\Big\|_{X^*}
    \le \sqrt{C} \Big\|\sum_{j=1}^\infty \xi_j e^*_j\Big\|_{X^*},
  \end{equation}
  for all $(\xi_j)\in c_{00}$.

  Then $Y=\overline{\spa(b_j:j\in\N)}$ is a complemented subspace of $X$ and
  \begin{equation*}
    P: X\to Y, \quad x\mapsto \sum_{n=1}^\infty b^*_n(x) b_n
  \end{equation*}
  is well defined and a bounded projection onto $Y$ with $\|P\|\le\lambda C$.  Moreover, if $(e_n)$
  is shrinking, then $\|P\|\le C$.
\end{lem}

\begin{proof}
  If $(e_n)$ is shrinking then $\spa(e^*_n:n\in\N)$ is norm dense in $X^*$ and therefore we have in
  that case
  \begin{equation*}
    \|x\|=\sup_{x^*\in\spa(e^*_n:n\in\N)} x^*(x).
  \end{equation*}
  If $(e_n)$ is a general basis whose basis constant is $\lambda$, we denote by $P_n$ the projection
  $P_n:X\to X$, $\sum_{j=1}^\infty x_j e_j \mapsto \sum_{j=1}^n x_j e_j$, and since
  $\|P_n\|=\|P^*_n\|\le \lambda$ we deduce for $x\in X$
  \begin{equation}\label{E:1.2.1}
    \begin{aligned}
      \|x\|&=\lim_{n\to\infty}\sup_{x^*\in B_{X^*}} x^*(P_n(x))
      \le \sup_{n\in\N, x^*\in B_{X^*}} {P^*_n(x^*)}(x)\\
      &\le \sup_{z^*\in \lambda B_{X^*}\cap\spa(e^*_j:j\in\N)} z^*(x) =\lambda \sup_{z^*\in
        B_{X^*}\cap\spa(e^*_j:j\in\N) } z^*(x).
    \end{aligned}
  \end{equation}

  If $x\in \spa(e_j)$ then $P(x)$ is a finite linear combination of elements of $(b_n)$.

  We compute:
  \begin{align*}
    \sup_{x\in B_X\cap\spa(e_j:j\in\N)} \|P(x)\|
    &=\sup_{x\in B_X\cap\spa(e_j:j\in\N)}\Big\|\sum_{j=1}^\infty b^*_j(x) b_j\Big\|\\
    &\le \sqrt{C}\sup_{x\in B_X\cap\spa(e_j:j\in\N)}\Big\|\sum_{j=1}^\infty b^*_j(x) e_j\Big\|.
  \end{align*}
  Using~\eqref{E:1.2.1} yields
  \begin{align*}
    \sup_{x\in B_X\cap\spa(e_j:j\in\N)} \|P(x)\|
    &\leq \lambda  \sqrt{C}\sup_{\substack{x\in B_X\cap\spa(e_j:j\in\N)\\x^*\in B_{X^*}\cap\spa(e^*_j:j\in\N)}}
    x^*\Big(\sum_{j=1}^\infty b^*_j(x) e_j\Big)\\
    &= \lambda\sqrt{C}\sup_{x^*\in B_{X^*}\cap\spa(e^*_j:j\in\N)} \Big\|\sum_{j=1}^\infty x^*(e_j)b_j^*\Big\|.
  \end{align*}
  By~\eqref{eq:basis-estimates:1}, we obtain therefore
  \begin{equation*}
    \sup_{x\in B_X\cap\spa(e_j:j\in\N)} \|P(x)\|
    \leq \lambda \sqrt{C}\sup_{(\xi_j)\in c_{00}, \|\sum \xi_j e_j^*\|\le 1}
    \Big\|\sum_{j=1}^\infty\xi_j b^*_j \Big\|
    \le\lambda C.
  \end{equation*}
  In the case that $(e_j)$ is shrinking we can replace in the first inequality $\lambda$ by $1$, and
  therefore obtain that $\|P(x)\|\le C$ for $x\in B_X$.
\end{proof}

\begin{prop}\label{P:1.3}
  Assume that $X$ is a Banach space with a basis $(e_n)$ and biorthogonal functionals $(e_n^*)$.
  Let $(b_n)$ and $(b^*_n)$ be block bases of $(e_n)$ and $(e^*_n)$, respectively, so that
  $b_m^*(b_n)=\delta_{m,n}$, for $m,n\in\N$, and so that for some $C\ge 1$ it follows that
  \begin{equation}\label{eq:basis-estimates:2}
    \Big\|\sum_{j=1}^\infty \xi_j b_j\Big\|_X
    \le \sqrt{C}\Big\|\sum_{j=1}^\infty \xi_j e_j\Big\|_X
    \qquad\text{and}\qquad
    \Big\|\sum_{j=1}^\infty \xi_j b^*_j\Big\|_{X^*}
    \le \sqrt{C} \Big\|\sum_{j=1}^\infty \xi_j e^*_j\Big\|_{X^*},
  \end{equation}
  for all $(\xi_j)\in c_{00}$.

  Then $(b_n)$ is $\lambda C$-impartially equivalent to $(e_n)$ and $(b^*_n)$ is $C$-impartially
  equivalent to $(e_n)$. If $(e_n)_n$ is Shrinking then $(b_n)$ is $C$-impartially equivalent to
  $(e_n)$.
\end{prop}
\begin{proof}
  For a sequence $(\xi_j)\in c_{00}$ we compute
  \begin{align*}
    \Big\|\sum_{j=1}^\infty \xi_j b_j\Big\|
    &\ge \sup_{(\eta_j)\in c_{00},\|\sum_{j=1}^\infty \eta_j b^*_j\|\le 1}\sum_{j=1}^\infty \eta_j b^*_j
      \Big(\sum_{j=1}^\infty \xi_j b_j\Big)\\
    &= \sup_{(\eta_j)\in c_{00},\|\sum_{j=1}^\infty \eta_j b^*_j\|\le 1}\sum_{j=1}^\infty \xi_j\cdot \eta_j.
  \end{align*}
  By~\eqref{eq:basis-estimates:2} and then \eqref{E:1.2.1} we obtain
  \begin{align*}
    \Big\|\sum_{j=1}^\infty \xi_j b_j\Big\|
    &\ge \sup_{(\eta_j)\in c_{00},\|\sum_{j=1}^\infty \eta_j e^*_j\|\le 1/\sqrt{C}}\sum_{j=1}^\infty \xi_j\cdot \eta_j\\
    &=\frac1{\sqrt{C}}\sup_{(\eta_j)\in c_{00},\|\sum_{j=1}^\infty \eta_j e^*_j\|\le 1}\sum_{j=1}^\infty \xi_j\cdot \eta_j
      =\geq\frac1{\lambda \sqrt{C}}\Big\|\sum_{j=1}^\infty \xi_j e_j\Big\|.
  \end{align*}
  Similarly we show that
  \begin{align*}
    \Big\|\sum_{j=1}^\infty \xi_j b^*_j\Big\|
    &\ge\frac1{\sqrt C}\Big\|\sum_{j=1}^\infty \xi_j e^*_j\Big\|.
      \qedhere
  \end{align*}
\end{proof}

In order to deduce that strategical reproducibility implies the factorization property, we will also
need a condition on diagonal operators which is automatically satisfied in the case that the given
basis is unconditional.

\begin{defin}
  \label{diagonal rip}
  Let $X$ be a Banach space with a normalized Schauder basis $(e_n)_n$. We say that the basis
  $(e_n)_n$ has the {\em uniform diagonal factorization property} if for every $\delta>0$ there
  exists $K(\delta)\geq1$ so that for every bounded diagonal operator $T:X\to X$ with
  $\inf_n \big|e_n^*(T(e_n))\big| \geq \delta$ the identity almost $K(\delta)$-factors through
  $T$. If we wish to be more specific we shall say that $(e_n)_n$ has the {\em $K(\delta)$-diagonal
    factorization property}.
\end{defin}

Note that if $(e_n)$ is unconditional then it has the uniform diagonal factorization property. The
following definition quantifies the uniform factorization property.

\begin{defin}
  \label{rip}
  Let $X$ be a Banach space with a normalized Schauder basis $(e_n)_n$. We say that the basis
  $(e_n)_n$ has the {\em uniform factorization property} if for every $\delta>0$ there exists
  $K(\delta)\geq 1$ so that for every bounded linear operator $T:X\to X$ with
  $\inf_n \big|e_n^*(T(e_n))\big| \geq \delta$ the identity almost $K(\delta)$-factors through
  $T$. If we wish to be more specific we shall say that $(e_n)_n$ has the {\em
    $K(\delta)$-factorization property}.
\end{defin}

\begin{remark}
  Note that in Definitions \ref{diagonal rip} and \ref{rip} $K(\delta)\geq1/\delta$. This can be
  witnessed by taking $T = \delta I$. Also whenever the function $K:(0,\infty)\to \mathbb{R}$ is
  well defined it is also continuous. In fact, for $0<\varepsilon<\delta$ if a simple scaling
  argument yields
  \begin{equation}
    \label{k-ontinuous}
    K(\delta) \leq K(\delta - \varepsilon) \leq \frac{\delta}{\delta - \varepsilon}K(\delta).
  \end{equation}
  To see this, use the following trick: if $T$ has a diagonal whose elements are absolutey bounded
  below by $\delta - \varepsilon$ then the elements of the diagonal of
  $\delta/(\delta - \varepsilon)T$ are absolutey bounded below by $\delta$. Inspecting
  \eqref{k-ontinuous} we also deduce $K(\delta)\leq K(1)/\delta$ and therefore
  \[\frac{1}{\delta}\leq K(\delta) \leq \frac{K(1)}{\delta},\]
  for all $\delta>0$.
\end{remark}

\begin{thm}\label{thm:strat-rep:1}
  \label{fun and games factor through one's identity}
  Let $X$ be a Banach space with a Schauder basis $(e_n)_n$ that has a basis constant
  $\lambda$. Assume also that
  \begin{itemize}

  \item[(i)] the basis $(e_i)_i$ has the $K(\delta)$-diagonal factorization property and

  \item[(ii)] the basis $(e_i)_i$ is $C$-strategically reproducible in $X$.

  \end{itemize}
  Then $(e_i)_i$ has the $ \lambda C^{2}K(\delta)$-factorization property.
\end{thm}

\begin{remark}
  It is worth pointing out that in Definition \ref{rip} the norm of $T$ does not appear and the
  factorization constant of the identity through $T$ depends only on the diagonal of $T$. This means
  that having the uniform factorization property is formally stronger than having the factorization
  property. Theorem \ref{thm:strat-rep:1} yields the stronger property. It is unclear whether these
  two properties are actually distinct.
\end{remark}

We postpone the proof of Theorem \ref{fun and games factor through one's identity} to first present
a necessary lemma.

\begin{lem}
  \label{what you need to factor}
  Let $X$ be a Banach space with a normalized Schauder basis $(e_n)_n$ with a basis constant
  $\lambda$, let $T:X\to X$ be a bounded linear operator, let $(x_n)_n$, $(x_n^*)_n$ be sequences in
  $X$ and $X^*$ respectively and let $\eta >0$, $C\geq 1$. Assume that the following are satisfied:
  \begin{itemize}

  \item[(i)] $(x_n)_n$ and $(e_n)_n$ are impartially $C$-equivalent,
  \item[(ii)] $(x^*_n)_n$ and $(e^*_n)_n$ are impartially $C$-equivalent,
  \item[(iii)] there exists a block sequence $(\tilde x_n^*)_n$ of $(e_n^*)$ so that
    $\sum_n\|x_n^*-\tilde x^*_n\| <\infty$,
  \item[(iv)] $\sum_n\sum_{m\neq n}|x^*_n(T(x_m))| < \eta$.
  \end{itemize}
  Then the diagonal operator $D:X\to X$ given by $D(e_n) = x_n^*(T(x_n))e_n$ is bounded and there
  exist bounded linear operators $B,A:X\to X$ with $\|A\|\|B\| \leq \lambda C^2$ so that
  $\|D - BTA\| < 2\lambda\eta$.

  If we additionally assume that $K\geq 1$ is such that the identity $K$-factors through $D$ and
  $\eta < 1/(2\lambda K)$ then the identity $\Big(\frac{KC^{2}}{1-2\lambda K\eta}\Big)$-factors
  through $T$.
\end{lem}

\begin{proof}
  The maps $A:X\to X$, $S:[x_n:n\kin\N]\to X$ with $A(e_n) = x_n$, $S(x_n) = e_n$ are well defined
  and satisfy $\|A\|\|S\| \leq C$.
  
  From Lemma \ref{L:1.2} it follows that the map $R: X\to [x_n:n\in\N]$ given by
  $R(x) = \sum_{n=1}^\infty x_n^*(x)x_n$ is well defined and $\|R\| \leq \lambda C$.
  
  Define $B = S\circ R:X\to X$. Then $\|B\circ A\| \leq \lambda C^{2}$. It also follows that for
  each $m\in\N$ we have $BTA(e_m) = \sum_{n=1}^\infty x_n^*(T(x_m))e_n$. By (iv) we deduce that for
  each $m\in\N$ we have
  \begin{equation*}
    \|BTA(e_m) - x_m^*(T(x_m))e_m\| \leq \sum_{n\neq m}\big|x_n^*(T(x_m))\big|.
  \end{equation*}
  Combining this with the triangle inequality we obtain that the desired diagonal map $D$ is bounded
  and $\|D - BTA\| < 2\lambda\eta$.

  For the additional part, assume that $\hat B:X\to X$ and $\hat A:X\to X$ are such that
  $\|\hat B\|\|\hat A\| \leq K$ and $I = \hat BD\hat A$. It follows that
  $\|I - \hat BBTA\hat A\|= \|\hat B(D-BTA)\hat A\| < 2\lambda K\eta < 1$. Hence, the map
  $Q = \hat BBTA\hat A$ is invertible with $\|Q^{-1}\| \leq 1/(1 - 2\lambda K\eta)$. In conclusion,
  if we set $\tilde B = Q^{-1}\hat BB$, $\tilde A = A\hat A$ then $\tilde BT\tilde A = I$ and
  $\|\tilde B\|\|\tilde A\| \leq \lambda K C^{2}/(1-2\lambda K\eta)$.
\end{proof}

\begin{proof}[Proof of Theorem \ref{fun and games factor through one's identity}]
  Let $\delta>0$, $T:X\to X$ be a bounded linear operator and
  $\inf_n \big|e_n^*(T(e_n))\big| \geq \delta$. Let us fix $\eta>0$ to be determined later.

  We will now describe a strategy for player (I) in a game $\mathrm{Rep}_{(X,(e_i))}(C,\eta)$, and
  assume player (II) answers by following his winning strategy.

  At the beginning, as player (I) chooses $N_1 = \{n\in\N:e_n^*(T(e_n)) \geq \delta\}$ and
  $N_2 = \{n\in\N: e_n^*(T(e_n)) \leq -\delta\}$. In the first step of the $n$'th turn he chooses
  $\eta_n < \eta(\|T\|n2^n\sqrt{C+\eta})^{-1}$ and if $l_n = \max_{1\leq k<n}(E_k)$ chooses
  $G_n = A_n^\perp$ and $W_n = (B_n)_\perp$, where
  \begin{equation*}
    \begin{split}
      A_n = \big\{ x_1,T(x_1),\ldots,x_{n-1},T(x_{n-1}),e_1, T(e_1),\ldots,e_{l_n}, T(e_{l_n})
      \big\},\\
      B_n = \big\{ x_1^*,T^*(x_1),\ldots,x_n^*,T^*(x_n^*),e^*_1,T^*(e^*_1),\ldots,e^*_{l_n},
      T^*(e^*_{l_n}) \big\}.
    \end{split}
  \end{equation*}
  Player (II), following a winning strategy, chooses $i_n = 1$ or $i_n = 2$, picks
  $E_n\subset N_{i_n}$, and non-negative scalars $(\lambda_i^{(n)})_{i\in E_n}$,
  $(\mu_i^{(n)})_{i\in E_n}$ with
  \begin{equation*}
    1-\eta
    < \sum_{i\in E_n}\lambda_i^{(n)}\mu_i^{(n)}
    < 1+\eta.
  \end{equation*}
  Then player (I), pick signs $(\varepsilon_i^{(n)})_{i\in E_n}$ so that if
  $x_n^* = \sum_{i\in E_n}\mu_i^{(n)}\varepsilon_i^{(n)}e_i^*$ then
  \begin{equation}
    \label{probabilities and stuff}
    \Big|x_n^*\Big(T\Big(\sum_{i\in E_n}\varepsilon_i^{(n)}\lambda_i^{(n)}e_i\Big)\Big)\Big| > (1-\eta)\delta.
  \end{equation}
  That this is possible follows using the following probabilistic argument:

  Let $r\keq(r_j)_{j\in E_n}$ be a Rademacher sequence, {\it i.e.}  $r_j$, $j\kin E_n$, are
  independent random variables on some probability space $(\Omega,\Sigma,\P)$, with
  $\P(r_j=1)\keq\P(r_j=-1)\keq\frac12$.
  \begin{align*}
    \E\Bigg(\Big(\sum_{i\in E_n} r_i \mu^{(n)}_i e^*_i\Big)
    \Big(T\Big(\sum_{j\in E_n} r_j\lambda^{(n)}_j e_j\Big)\Big)\Bigg)
    &=\E\Big( \sum_{i,j\in E_n} r_i r_j  \mu^{(n)}_i\lambda^{(n)}_j e^*_j(T(e_i))\Big)\notag\\
    & =\sum_{i\in E_n} \mu^{(n)}_i\lambda^{(n)}_ie^*_i(T(e_i))> \delta(1-\eta).\notag 
  \end{align*}
  It follows therefore that we can choose $(\vp^{(n)}_j)_{j\in E_n}$ appropriately to satisfy
  \eqref{probabilities and stuff}.

  After the game is completed, put
  \begin{equation*}
    x_n = \sum_{i\in E_n}\varepsilon_i^{(n)}\lambda_i^{(n)}e_i
    \qquad\text{and}\qquad
    x_n^* = \sum_{i\in E_n}\varepsilon_i^{(n)}\mu_i^{(n)}e_i^*.
  \end{equation*}
  Conditions (i) to (iv) of Definition \ref{D:1.3} are satisfied. Then \eqref{probabilities and
    stuff} can be rewritten as
  \begin{equation}
    \label{diagonal big and mighty}
    |x^*_n(T(x_n))| \geq (1-\eta)\delta.
  \end{equation}
  Furthermore, observe that for any $k<n$ we have
  \begin{equation*}
    \begin{split}
      |x_k^*(T(x_n))| &= |T^*(x_k^*)(x_n)|
      \leq \|T^*(x_k^*)\| \cdot\mathrm{dist}(x_n,W_n)\\
      &\leq \|T\|\sqrt{C+\eta}\cdot \mathrm{dist}(x_n,W_n)\\
      |x_n^*(T(x_k))| &\leq \|T(x_k)\|\cdot\mathrm{dist}(x_n^*,G_n)
      \leq\|T\|\sqrt{C+\eta}\cdot\mathrm{dist}(x_n^*,G_n).
    \end{split}
  \end{equation*}
  We conclude that $\sum_n\sum_{m\neq n}|x^*_n(T(x_m))| < \eta$. A similar argument yields that
  $(x_n^*)_n$ is summably close to a block sequence of $(e_i^*)_i$. By Lemma \ref{what you need to
    factor}, the diagonal operator $D:X\to X$ given by $De_n = x_n^*(T(x_n))$ is bounded. By
  assumption, for any $\xi>0$, the identity $(K(\delta - \eta)+\xi)$ factors through $D$ and if
  $\eta$ is sufficiently small then by the second part of Lemma \ref{what you need to factor} the
  identity
  $\Big(\frac{ (\lambda K(\delta-\eta)+\xi)(C+\eta)^{2}}{1-2\lambda (K(\delta -
    \eta)+\xi)\eta}\Big)$-factors through $T$. Recall that by \eqref{k-ontinuous} the function
  $K:(0,\infty)\to\mathbb{R}$ is continuous. As we could have picked $\eta$ and $\xi$ arbitrarily
  close to zero we deduce that the identity almost $\lambda K(\delta)C^{2}$-factors through $T$.
\end{proof}

\section{A basic overview: multi-parameter Lebesgue and Hardy spaces}
\label{sec:basic-overview-multi}

Here we give a preparation for the following sections in which we exhibit examples of strategically
reproducible bases.

\subsection{The multi-parameter Haar system}
\label{sec:multi-parameter-haar}

We denote by $\mathcal{D}$ the collection of all dyadic intervals in $[0,1)$, namely
\begin{equation*}
  \mathcal{D} = \Big\{\Big[\frac{i-1}{2^j},\frac{i}{2^j}\Big): j\in\mathbb{N}\cup\{0\}, 1\leq
  i\leq 2^j\Big\}.
\end{equation*}
For each $n\in\mathbb{N}\cup\{0\}$ we define $\mathcal{D}_n = \{I\in\mathcal{D}: |I| = 2^{-n}\}$ and
$\mathcal{D}^n = \cup_{k=0}^n\mathcal{D}_k$.  We define the bijective function
$\mathcal{O} : \mathcal{D}\to \mathbb{N}$ by
\begin{equation*}
  \Big[\frac{i-1}{2^j},\frac{i}{2^j}\Big)
  \mapsto 2^j + i - 1.
\end{equation*}
The function $\mathcal{O}$ defines a linear order on $\mathcal{D}$.  Recall that Haar system
$(h_I)_{I\in\mathcal{D}}$ is defined as follows: if $I = [(i-1)/2^j,i/2^j)$ then set
$I^+ = [(i-1)/2^j,(2i-1)/2^{j+1})$, $I^- = [(2i-1)/2^{j+1},i/2^j)$, and
\begin{equation*}
  h_I
  = \chi_{I^+} - \chi_{I^-}.
\end{equation*}

The $d$-parameter dyadic rectangles $\mathcal{R}_d$ are given by
\begin{equation*}
  \mathcal{R}_d = \{ I_1\times\dots\times I_d : I_1,\dots,I_d\in\mathcal{D}\},
\end{equation*}
and the $d$-parameter tensor product Haar system $(h_{\bar I})_{\bar I \in\mathcal{R}_d}$ is given
by
\begin{equation*}
  h_{\bar I}(t_1,t_2,\ldots t_d)
  =h_{I_1}(t_1)\cdot  h_{I_2}(t_2)\cdot \ldots \cdot  h_{I_d}(t_d),
  \qquad t_1,t_2,\ldots t_d\in [0,1).
\end{equation*}

\subsection{A linear order on $\mathcal{R}_2$}\label{sec:order-clpc}

First, we define the bijective function $\mathcal{O}_{\mathbb N_0^2} : \mathbb N_0^2\to \mathbb N_0$
by
\begin{equation*}
  \mathcal O_{\mathbb N_0^2}(m,n) =
  \begin{cases}
    n^2 + m, & \text{if $m < n$},\\
    m^2 + m + n, & \text{if $m \geq n$}.
  \end{cases}
\end{equation*}
To see that $\mathcal O_{\mathbb N_0^2}$ is bijective consider that for each $k\in \mathbb N$:
\begin{itemize}
\item $\mathcal{O}_{\mathbb{N}_0^2}(0,0)=0$,
\item $m\mapsto \mathcal O_{\mathbb N_0^2}(m,k)$ maps $\{0,\dots,k-1\}$ bijectively onto
  $\{k^2,\dots,k^2+k-1\}$ and preserves the natural order on $\mathbb N_0$,
\item $\mathcal O_{\mathbb N_0^2}(k,0) = \mathcal O_{\mathbb N_0^2}(k-1,k)+1$,
\item $n\mapsto \mathcal O_{\mathbb N_0^2}(k,n)$ maps $\{0,\dots,k\}$ bijectively onto
  $\{k^2+k,\dots,k^2+2k\}$ and preserves the natural order on $\mathbb N_0$,
\item $\mathcal{O}_{\mathbb{N}_0^2}(0,k+1)=\mathcal{O}_{\mathbb{N}_0^2}(k,k)+1$.
\end{itemize}

Now, let $\lesslex$ denote the lexicographic order on $\mathbb R^3$. For two dyadic rectangles
$I_k\times J_k\in \mathcal{R}_2$ with $|I_k|=2^{-m_k}$, $|J_k|=2^{-n_k}$, $k=0,1$, we define
$I_0\times J_0 \drless I_1\times J_1$ if and only if
\begin{equation*}
  \big( \mathcal O_{\mathbb N_0^2}(m_0,n_0),\inf I_0, \inf J_0 \big)
  \lesslex \big( \mathcal O_{\mathbb N_0^2}(m_1,n_1),\inf I_1, \inf J_1 \big).
\end{equation*}
Associated to the linear ordering $\drless$ is the bijective index function
$\drindex : \mathcal{R}_2 \rightarrow \mathbb N_0$ defined by
\begin{equation*}
  \drindex(R_0) < \drindex(R_1)
  \Leftrightarrow R_0 \drless R_1,
  \qquad R_0,R_1 \in \mathcal{R}_2.
\end{equation*}
See Figure~\ref{fig:ordering-relation:2} for a picture of $\drindex$.
\begin{figure}[bt]
  \begin{center}
    \includegraphics{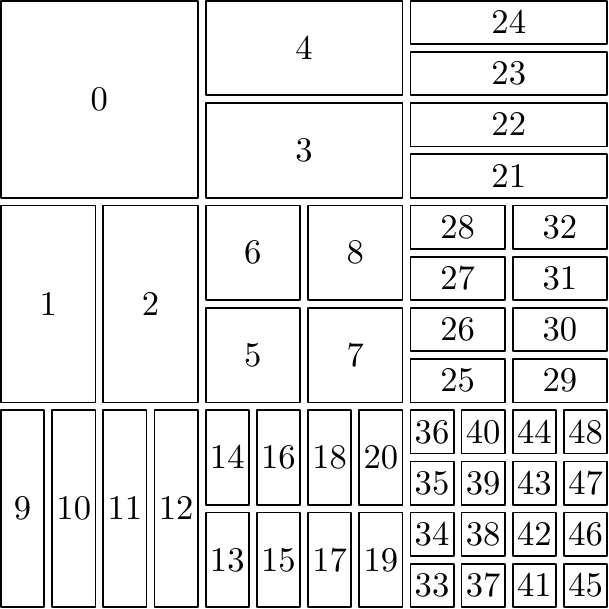}
  \end{center}
  \caption{The first $49$ rectangles and their indices $\drindex$.}
  \label{fig:ordering-relation:2}
\end{figure}

\subsection{Multi-parameter Lebesgue spaces}
\label{sec:multi-param-lebesg}

Given $1 \leq p \leq \infty$, we define $L^p_0$ as the closed subspace of $L^p[0,1]$ given by
\begin{equation*}
  L^p_0 = \Bigl\{ f\in L^p : \int_0^1 f(t) dt = 0\Bigr\}.
\end{equation*}
Note that the Haar system $(h_I)_{I\in\mathcal{D}}$ ordered by $\mathcal{O}$ is a monotone basis for
$L^p_0$, whenever $1 \leq p < \infty$.  Moreover, $(h_I)_{I\in\mathcal{D}}$ is unconditional in
$L^p_0$, whenever $1 < p < \infty$. The only reason for considering $L^p_0$ rather than $L^p$, is a
notational one.  Otherwise we would have to consider the first basis element of $L^p$, namely
$\chi_{[0,1]}$, always separately from the other ones.

Given $\bar p = (p_1,\dots,p_d)$, where $1 \leq p_1,\dots,p_d\leq \infty$ we define the mixed norm
Lebesgue space $L^{\bar p}$ by
\begin{equation*}
  L^{\bar p} = L^{p_1}(L^{p_2}(L^{p_3}(\ldots(L^{p_d} ))\ldots )).
\end{equation*}
Moreover, we define the closed subspace $L^{\bar p}_0$ of $L^{\bar p}$ by
\begin{equation*}
  L^{\bar p}_0
  = \Bigl\{
  f\in L^{\bar p} : \int_0^1 f(t_1\dots,t_d) d t_j = 0\ \text{for all $1\leq j\leq d$}
  \Bigr\}.
\end{equation*}
The $d$-parameter tensor product Haar system $(h_{\bar I})_{\bar I \in\mathcal{R}_d}$ is an
unconditional basis for $L^{\bar p}_0$, whenever $\bar p = (p_1,\dots,p_d)$ and
$1 < p_1,\dots,p_d < \infty$.  The dual of $L^{\bar p}_0$ is then $L^{\bar q}_0$, where
$\bar q=(q_1,\ldots, q_d)$, and $\frac1{p_i}+\frac1{q_i}=1$, for $i=1,2,\ldots,d$.

In the following, we prove the equivalence of the $L^{\bar p}$ norm and the d-parametric square
function norm in the reflexive case, i.e. $p = (p_1,\dots,p_d)$ with $1 < p_1,\dots,p_d < \infty$.
The content of Proposition~\ref{P:2.6} was known and used by Capon~\cite{capon:1982:2}; for the
convenience of the reader, we provide a detailed exposition, below.
\begin{prop}\label{P:2.6}
  Let $d\in\N$, $1<p_1,p_2, \ldots, p_d<\infty $ and $\bar p=(p_1,p_2,\ldots,p_d)$. For
  $f=\sum_{ \bar I\in \mathcal{R}_d} a_{\bar I} h_{\bar I}\in L^{\bar p}_0$ we define
  \begin{equation}\label{E:2.6.1}
    \trin f \trin_{\bar p}
    =\Bigg(\int_0^1\Bigg(\ldots\int_0^1 \Bigg(
    \int_0^1\Big(
    \sum_{\bar I\in\mathcal{R}_d} a_{\bar I}^2 h^2_{\bar I}(t_1,\ldots t_d)
    \Big)^{\frac{p_{d}}2} dt_d
    \Bigg)^{\frac{p_{d-1}}{p_{d}}}\ldots
    \Bigg)^{\frac{p_1}{p_2}}dt_1\Bigg)^{\frac1{p_1}}.
  \end{equation}
  Then $\trin\cdot\trin_{\bar p}$ is an equivalent norm on $L^{\bar p}_0$.  The dual norm to
  $\trin\cdot\trin_{\bar p}$ is equivalent (with constants depending on $\bar p$) to
  $\trin\cdot\trin_{\bar q}$, where $\bar q=(q_1,q_2,\ldots q_d)$, and $\frac1{p_i}+\frac1{q_i}=1$,
  for $i=1,2,\ldots,d$.
\end{prop}

\begin{proof}
  We show~\eqref{E:2.6.1} by induction on $d$.  By the unconditionality of the Haar system in $L^p$,
  $1 < p < \infty$ (theorem of Paley-Marcinkiewicz~\cite{Paley1932} and~\cite{marcinkiewicz:1937})
  the statement is true for $d=1$.  Assume that we proved~\eqref{E:2.6.1} for
  $L^{(p_2,\ldots,p_d)}_0$.  In the following, we use the abbreviations
  $X_j = L^{(p_1,\ldots,p_{j})}_0$ and $Y_j = L^{(p_j,\ldots,p_d)}_0(\ell^2(\mathcal{R}_{d-1}))$,
  $1\leq j \leq d$.  Since $L^{(p_2,\ldots,p_d)}_0$ has the $\mathrm{UMD}$ property\cite[page
  II.12]{maurey:systeme:1975} we obtain that
  \begin{equation*}
    \|f\|_{L^{\bar p}_0}
    \sim_{\bar p} \biggl(\int_0^1
    \Bigl\| \sum_{\bar I} \varepsilon_{I_1} a_{\bar I} h_{I_1}(t_1) h_{I_2,\ldots,I_d}\Bigr\|_{(p_2,\ldots,p_d)}^{p_1}
    dt_1 \biggr)^{\frac{1}{p_1}},
  \end{equation*}
  for all choices of signs $\varepsilon_{I_1}$, $I_1\in\mathcal{D}$.  Hence, by averaging and
  Kahane's inequality~\cite[Theorem~4]{kahane:1985}, we obtain
  \begin{align*}
    \|f\|_{L^{\bar p}_0}
    &\sim_{\bar p} \biggl(  \int_0^1 \Bigl( \cond_{\varepsilon} \Bigl\|
      \sum_{\bar I} \varepsilon_{I_1} a_{\bar I} h_{I_1}(t_1) h_{I_2,\ldots,I_d}
      \Bigr\|_{(p_2,\ldots,p_d)} \Bigr)^{p_1}
      dt_1 \biggr)^{\frac{1}{p_1}}\\
    &= \biggl\| \cond_{\varepsilon} \Bigl\|
      \sum_{\bar I} \varepsilon_{I_1} a_{\bar I} h_{I_1} h_{I_2,\ldots,I_d}
      \Bigr\|_{(p_2,\ldots,p_d)} \biggr\|_{X_1}.
  \end{align*}
  By induction hypothesis, we obtain
  \begin{align*}
    \|f\|_{L^{\bar p}_0}
    &\sim_{\bar p} \biggl\| \cond_{\varepsilon} \Bigl\| \Bigl(
      \sum_{(I_2,\ldots,I_d)}\bigl( \sum_{I_1} \varepsilon_{I_1} a_{\bar I} h_{I_1} \bigr)^2
      h^2_{I_2,\ldots,I_d}
      \Bigr)^{\frac{1}{2}}
      \Bigr\|_{(p_2,\ldots,p_d)} \biggr\|_{X_1}.
  \end{align*}
  Now, observe that
  \begin{align*}
    &\cond_{\varepsilon} \Bigl\| \Bigl(
      \sum_{(I_2,\ldots,I_d)}\bigl( \sum_{I_1} \varepsilon_{I_1} a_{\bar I} h_{I_1} \bigr)^2
      h^2_{I_2,\ldots,I_d}
      \Bigr)^{\frac{1}{2}}
      \Bigr\|_{(p_2,\ldots,p_d)}\\
    &\qquad= \cond_{\varepsilon} \biggl\| \Bigl\|
      \sum_{I_1} \varepsilon_{I_1}
      \Bigl( a_{\bar I} h_{I_1} h_{I_2,\ldots,I_d} \Bigr)_{(I_2,\ldots,I_d)}
      \Bigr\|_{\ell^2(I_2,\ldots,I_d)} \biggr\|_{(p_2,\ldots,p_d)}\\
    &\qquad= \cond_{\varepsilon} \Bigl\|
      \sum_{I_1} \varepsilon_{I_1}
      \Bigl( a_{\bar I} h_{I_1} h_{I_2,\ldots,I_d} \Bigr)_{(I_2,\ldots,I_d)}
      \Bigr\|_{Y_2}.
  \end{align*}
  Kahane's inequality yields
  \begin{align*}
    &\cond_{\varepsilon} \Bigl\|
      \sum_{I_1} \varepsilon_{I_1}
      \Bigl( a_{\bar I} h_{I_1} h_{I_2,\ldots,I_d} \Bigr)_{(I_2,\ldots,I_d)}
      \Bigr\|_{Y_2}\\
    &\qquad\sim_{\bar p} \Bigl(
      \cond_{\varepsilon} \Bigl\|
      \sum_{I_1} \varepsilon_{I_1}
      \Bigl( a_{\bar I} h_{I_1} h_{I_2,\ldots,I_d} \Bigr)_{(I_2,\ldots,I_d)}
      \Bigr\|_{Y_2}^{p_2}
      \Bigr)^{\frac{1}{p_2}}\\
    &\qquad=\Bigl( \int \cond_{\varepsilon} \Bigl\|
      \sum_{I_1} \varepsilon_{I_1}
      \Bigl( a_{\bar I} h_{I_1} h_{I_2,\ldots,I_d} \Bigr)_{(I_2,\ldots,I_d)}
      \Bigr\|_{Y_3}
      dt_2
      \Bigr)^{\frac{1}{p_2}}.
  \end{align*}
  Combining our estimates yields
  \begin{align*}
    \|f\|_{L^{\bar p}_0}
    &\sim_{\bar p} \biggl\| \Bigl(
      \int \cond_{\varepsilon} \Bigl\|
      \sum_{I_1} \varepsilon_{I_1}
      \Bigl( a_{\bar I} h_{I_1} h_{I_2,\ldots,I_d} \Bigr)_{(I_2,\ldots,I_d)}
      \Bigr\|_{Y_3}
      dt_2 \Bigr)^{\frac{1}{p_2}} \biggr\|_{X_1}\\
    &= \biggl\|
      \cond_{\varepsilon} \Bigl\|
      \sum_{I_1} \varepsilon_{I_1}
      \Bigl( a_{\bar I} h_{I_1} h_{I_2,\ldots,I_d} \Bigr)_{(I_2,\ldots,I_d)}
      \Bigr\|_{Y_3}
      \biggr\|_{X_2}.
  \end{align*}
  With the same argument, we obtain
  \begin{equation*}
    \|f\|_{L^{\bar p}_0}
    \sim_{\bar p}
    \biggl\| \cond_{\varepsilon} \Bigl\|
    \sum_{I_1} \varepsilon_{I_1}
    \Bigl( a_{\bar I} h_{I_1} h_{I_2,\ldots,I_d} \Bigr)_{(I_2,\ldots,I_d)}
    \Bigr\|_{Y_4}
    \biggr\|_{X_3}.
  \end{equation*}
  Continuing in this fashion yields
  \begin{equation}\label{eq:iter:1}
    \|f\|_{L^{\bar p}_0}
    \sim_{\bar p}
    \biggl\| \cond_{\varepsilon}\Bigl\|
    \sum_{I_1} \varepsilon_{I_1}
    \Bigl( a_{\bar I} h_{I_1} h_{I_2,\ldots,I_d} \Bigr)_{(I_2,\ldots,I_d)}
    \Bigr\|_{\ell^2(\mathcal{R}_{d-1})}
    \biggr\|_{L^{\bar p}_0}.
  \end{equation}
  Applying Kahane's inequality one last time, we obtain
  \begin{align}
    &\cond_{\varepsilon}\Bigl\|
      \sum_{I_1} \varepsilon_{I_1}
      \Bigl( a_{\bar I} h_{I_1} h_{I_2,\ldots,I_d} \Bigr)_{(I_2,\ldots,I_d)}
      \Bigr\|_{\ell^2(\mathcal{R}_{d-1})}
      \notag\\
    &\qquad \sim_{\bar p} \biggl( \cond_{\varepsilon}\Bigl\|
      \Bigl( \sum_{I_1} \varepsilon_{I_1} a_{\bar I} h_{I_1} h_{I_2,\ldots,I_d} \Bigr)_{(I_2,\ldots,I_d)}
      \Bigr\|_{\ell^2(\mathcal{R}_{d-1})}^2
      \biggr)^{1/2}.
      \label{eq:iter:2}
  \end{align}
  Note that the last expression is equal to
  \begin{align}\label{eq:iter:3}
    \biggl( \sum_{I_2,\ldots,I_d} \cond_{\varepsilon} \Bigl|
    \sum_{I_1} \varepsilon_{I_1} a_{\bar I} h_{I_1} h_{I_2,\ldots,I_d}
    \Bigr|^2
    \biggr)^{1/2}
    &= \biggl(
      \sum_{I_2,\ldots,I_d} \sum_{I_1} a_{\bar I}^2 h_{I_1}^2 h_{I_2,\ldots,I_d}^2
      \biggr)^{1/2}
      \notag\\
    &= \biggl( \sum_{\bar I} a_{\bar I}^2 h_{\bar I}^2 \biggr)^{1/2}.
  \end{align}
  Combining~\eqref{eq:iter:1} with \eqref{eq:iter:2} and~\eqref{eq:iter:3} yields~\eqref{E:2.6.1}.
\end{proof}

\subsection{Hardy spaces}
\label{sec:hardy-spaces}

We define the dyadic Hardy spaces $H^p$, $1\leq p < \infty$ as the completion of
\begin{equation*}
  \spa\{ h_I : I\in\mathcal{D}\}
\end{equation*}
under the square function norm
\begin{equation*}
  \Bigl\| \sum_{I\in\mathcal{D}} a_I h_I \Bigr\|_{H^p}
  = \Bigl\|\mathbb{S}\Bigl(\sum_{I\in\mathcal{D}} a_I h_I\Bigr)\Bigr\|_{L^p},
\end{equation*}
where the square function $\mathbb{S}$ is given by
\begin{equation*}
  \mathbb{S}\Bigl(\sum_{I\in\mathcal{D}} a_I h_I\Bigr)
  = \Bigl( \sum_{I\in\mathcal{D}} a_I^2 h_I^2 \Bigr)^{1/2},
\end{equation*}
for all scalar sequences $(a_I)_{I\in\mathcal{D}}$.

The bi-parameter dyadic Hardy spaces $H^p(H^q)$, $1\leq p,q <\infty$ are defined as the completion
of
\begin{equation*}
  \spa\{ h_{\bar I} : \bar I\in\mathcal{R}_2\}
\end{equation*}
under the bi-parameter square function norm
\begin{equation*}
  \Bigl\| \sum_{\bar I\in\mathcal{R}_2} a_{\bar I} h_{\bar I} \Bigr\|_{H^p(H^q)}
  = \Bigl\|\mathbb{S}\Bigl(\sum_{\bar I\in\mathcal{R}_2} a_{\bar I} h_{\bar I}\Bigr)\Bigr\|_{L^p(L^q)},
\end{equation*}
where the bi-parameter square function $\mathbb{S}$ is given by
\begin{equation*}
  \mathbb{S}\Bigl(\sum_{\bar I\in\mathcal{R}_2} a_{\bar I} h_{\bar I}\Bigr)
  = \Bigl( \sum_{\bar I\in\mathcal{R}_2} a_{\bar I}^2 h_{\bar I}^2 \Bigr)^{1/2},
\end{equation*}
for all scalar sequences $(a_{\bar I})_{\bar I\in\mathcal{R}_2}$.

The following Lemma is taken from~\cite{laustsen:lechner:mueller:2015}.
\begin{lem}\label{lem:w-w*-convergence}
  For $m\in\mathbb N$, let $\mathcal{X}_m$ and $\mathcal{Y}_m$ be non-empty, finite families of
  pairwise disjoint dyadic intervals, define
  $f_m = \sum_{I\in\mathcal{X}_m,\, J\in\mathcal{Y}_m} h_{I\times J}$, and let $1\leq p,q < \infty$.
  Suppose in addition that:
  \begin{itemize}
  \item $\mathcal{X}_m\cap\mathcal{X}_n=\emptyset$ or $\mathcal{Y}_m\cap\mathcal{Y}_n=\emptyset$
    whenever $m,n\in\mathbb N$ are distinct;
  \item $\bigcup\mathcal{X}_m = \bigcup\mathcal{X}_n$ and
    $\bigcup\mathcal{Y}_m = \bigcup\mathcal{Y}_n$ for all $m,n\in\mathbb{N}$.
  \end{itemize}
  Then for each $\gamma\in\ell^\infty(\mathcal{R})$ with $\|\gamma\|_\infty\leq 1$, the operator
  $M_\gamma$ defined as the linear extension of the map
  $h_{I\times J}\mapsto \gamma_{I\times J} h_{I\times J}$ is bounded by $1$, both as a map from
  $H^p(H^q)$ to itself and from $H^p(H^q)^*$ to itself.  Moreover,
  \begin{enumerate}[(i)]
  \item\label{lem:w-w*-convergence:1} for each $g\in H^p(H^q)^*$,
    $\sup_{\gamma\in B_{\ell^\infty(\mathcal{R})}}|\langle M_\gamma f_m, g\rangle|\to 0$ as
    $m\to\infty$;
  \item\label{lem:w-w*-convergence:2} for each $g\in H^p(H^q)$,
    $\sup_{\gamma\in B_{\ell^\infty(\mathcal{R})}}|\langle M_\gamma g, f_m\rangle|\to 0$ as
    $m\to\infty$.
  \end{enumerate}
\end{lem}

\subsection{Collections of dyadic intervals}
\label{sec:coll-dyad-interv}

We introduce convenient notation and gather basic facts of collections of dyadic intervals.

\begin{notation}~\label{level partition}
  \begin{itemize}

  \item[(i)] For $\mathcal{A}\subset\mathcal{D}$ set
    $\mathscr{G}_0(\mathcal{A}) = \{I\in\mathcal{A}: I$ is maximal with respect to inclusion$\}$.

  \item[(ii)] For $\mathcal{A}\subset\mathcal{D}$ recursively define for $n\in\mathbb{N}$ the
    collection
    \begin{equation*}
      \mathscr{G}_n(\mathcal{A}) =
      \mathscr{G}_0(\mathcal{A}\setminus(\cup_{k=0}^{n-1}\mathscr{G}_k(\mathcal{A}))).
    \end{equation*} For every
    $n\in\mathbb{N}$ and $I\in\mathscr{G}_{n+1}(\mathcal{A})$ there is a unique
    $J\in\mathscr{G}_{n}(\mathcal{A})$ with $I\subset J$. In paricular,
    $\mathscr{G}_{n+1}(\mathcal{A})^*\subset \mathscr{G}_{n}(\mathcal{A})^*$

  \item[(iii)] For $\mathcal{A}\subset\mathcal{D}$ set
    $\lim\sup\mathcal{A} = \cap_n \mathscr{G}_n(\mathcal{A})^*$.

  \item[(iv)] For a finite $\mathcal{H}\subset\mathcal{D}$, consisting of pairwise disjoint
    intervals, and $\bar \varepsilon = (\varepsilon)_{I\in\mathcal{H}}\in\{-1,1\}^\mathcal{H}$ we
    define
    \begin{equation*}
      \mathcal{H}^*_{\bar\varepsilon} = \Big[\sum_{I\in\mathcal{H}}\varepsilon_Ih_I = 1\Big]\text{
        and }\mathcal{H}^*_{\text{-}\bar\varepsilon} = \Big[\sum_{I\in\mathcal{H}}\varepsilon_Ih_I =
      -1\Big].
    \end{equation*}
    These two sets have measure $|\mathcal{H}^*|/2$ and they form a partition of $\mathcal{H}^*$.
  \item[(v)] For $\mathcal{A}\subset\mathcal{D}$, $n,k\in\mathbb{N}$, with $k\geq n$ and a finite
    $\mathcal{H}\subset\mathscr{G}_n(\mathcal{A})$ define the collection
    \begin{equation*}
      \mathcal{H}^\mathrm{succ}_k = \{I\in\mathscr{G}_k(\mathcal{A}): I\subset J\text{ for some
      }J\in \mathcal{H}\}.
    \end{equation*}
    For any $\bar \varepsilon\in\{-1,1\}^{\mathcal{H}}$ define the sets
    \begin{equation*}
      \mathcal{H}^\mathrm{succ}_{\bar \varepsilon, k} = \{I\in\mathscr{G}_k(\mathcal{A}): I\subset
      \mathcal{H}^*_{\bar\varepsilon}\}\text{ and }\mathcal{H}^\mathrm{succ}_{\text{-}\bar
        \varepsilon, k} = \{I\in\mathscr{G}_k(\mathcal{A}): I\subset
      \mathcal{H}^*_{\text{-}\bar\varepsilon}\}.
    \end{equation*}
    The sets $\mathcal{H}^\mathrm{succ}_{\bar \varepsilon, k}$,
    $\mathcal{H}^\mathrm{succ}_{\text{-}\bar \varepsilon, k}$ form a partition of
    $\mathcal{H}_k$. We point out that the definitions of $\mathcal{H}^\mathrm{succ}_k$,
    $\mathcal{H}^\mathrm{succ}_{\bar \varepsilon, k}$, and
    $\mathcal{H}^\mathrm{succ}_{\text{-}\bar \varepsilon, k}$ depend on the ambient collection
    $\mathcal{A}$.
  \end{itemize}
\end{notation}

Lemma~\ref{assume finite} and~\ref{eventual choices} will be used in
Section~\ref{sec:haar-system-l1_0-1}.
\begin{lem}\label{assume finite}
  Let $\mathcal{A}\subset\mathcal{D}$. Then for any $\kappa >0$ there is
  $\tilde{\mathcal{A}}\subset\mathcal{A}$ so that for each $k\in\mathbb{N}$ we have
  \begin{itemize}
  \item[(i)] $\mathscr{G}_k(\tilde{\mathcal{A}})$ is finite and
    $\mathscr{G}_k(\tilde{\mathcal{A}})\subset\mathscr{G}_k(\mathcal{A})$ and
  \item[(ii)] $|\limsup(\tilde{\mathcal{A}})| \geq |\limsup(\mathcal{A})| - \kappa$.
  \end{itemize}
\end{lem}
\begin{proof}
  Pick a finite $\mathcal{B}_0\subset\mathscr{G}_0(\mathcal{A})$ with
  $|\mathscr{G}_0(\mathcal{A})^*\setminus\mathcal{B}_0^*|<\kappa/2$. Recursively for
  $k\in\mathbb{N}$, if $\mathcal{C} = \mathcal{B}_{k-1}$, pick
  $\mathcal{B}_k\subset\mathcal{C}_k^\mathrm{succ}$ with
  $|(\mathcal{C}_k^\mathrm{succ})^*\setminus\mathcal{B}^*_{k}|<\kappa/2^{k+1}$. Set
  $\tilde{\mathcal{A}} = \cup_{k=0}^\infty\mathcal{B}_k$. One can check by induction that for all
  $k\in\mathbb{N}$ we have
  $\mathscr{G}_k(\tilde{\mathcal{A}}) = \mathcal{B}_k\subset\mathscr{G}_k(\mathcal{A})$ and that
  $|\mathscr{G}_k(\mathcal{A})^*\setminus\mathcal{B}_k^*| \leq \sum_{i=1}^k\kappa/2^i$.  The
  conclusion easily follows.
\end{proof}

\begin{lem}
  \label{eventual choices}
  Let $\mathcal{A}\subset\mathcal{D}$, $n\in\mathbb{N}$, $\kappa\in(0,1/2)$, and
  $\mathcal{H}\subset\mathscr{G}_n(\mathcal{A})$ be a non-empty finite collection so that if
  $A = \limsup\mathcal{A}$ then $|\mathcal{H}^*\cap A| >(1-\kappa)|\mathcal{H}^*|$. The following
  hold.
  \begin{itemize}
  \item[(i)] For any $\bar\varepsilon\in\{-1,1\}^\mathcal{H}$, if
    $C = \mathcal{H}_{\bar\varepsilon}^*$ or $C = \mathcal{H}_{\text{-}\bar\varepsilon}^*$ we have
    \begin{equation*}
      \frac{|\mathcal{H}^*|}{2}\geq |C\cap A| > (1-2\kappa)|C| =
      (1-2\kappa)\frac{|\mathcal{H}^*|}{2}.
    \end{equation*}
  \item[(ii)] For any $\delta > 0$ there exists $k_0\in\mathbb{N}$ so that for all $k\geq k_0$ and
    $\bar\varepsilon\in\{-1,1\}^\mathcal{H}$, if
    $C = (\mathcal{H}_{\bar\varepsilon,k}^\mathrm{succ})^*$ or
    $C = (\mathcal{H}_{\text{-}\bar\varepsilon,k}^\mathrm{succ})^*$ we have
    $|C\cap A| \geq (1-\delta)|C|$.
  \end{itemize}
\end{lem}

\begin{proof}
  For the proof of (i) let $C = \mathcal{H}_{\bar\varepsilon}^*$.  Recall that
  $|\mathcal{H}_{\bar\varepsilon}^*| = |\mathcal{H}_{\text{-}\bar \varepsilon}^*| =
  |\mathcal{H^*}|/2$. Now,
  \begin{equation*}
    \begin{split}
      (1-\kappa)2|\mathcal{H}_{\bar\varepsilon}^*| &=(1-\kappa)||\mathcal{H}^*| < |\mathcal{H}^*\cap
      A| = |\mathcal{H}_{\bar\varepsilon}^*\cap A| +
      |\mathcal{H}_{\text{-}\bar\varepsilon}^*\cap A|\\
      &\leq |\mathcal{H}_{\bar\varepsilon}^*\cap A| + |\mathcal{H}_{\text{-}\bar\varepsilon}^*| =
      |\mathcal{H}_{\bar\varepsilon}^*\cap A| + |\mathcal{H}_{\bar\varepsilon}^*|
    \end{split}
  \end{equation*}
  which yields
  $|\mathcal{H}_{\bar\varepsilon}^*\cap A| > (1-2\kappa)|\mathcal{H}_{\bar\varepsilon}^*|$. The same
  argument works for $C = \mathcal{H}_{\text{-}\bar\varepsilon}$.

  We now prove (ii). As there are finitely many choices of $\bar\varepsilon\in\{-1,1\}^\mathcal{H}$
  it suffices to prove it by fixing one of them. Observe the sequence
  $((\mathcal{H}_{\bar\varepsilon,k}^{\mathrm{succ}})^*)_k$ is decreasing so we can define
  $(\mathcal{H}_{\bar\varepsilon,\infty}^\mathrm{succ})^* =
  \cap_k(\mathcal{H}_{\bar\varepsilon,k}^{\mathrm{succ}})^* = \mathcal{H}_{\bar\varepsilon}^*\cap A
  = \cap_k((\mathcal{H}_{\bar\varepsilon,k}^{\mathrm{succ}})^*\cap A)$. For one we obtain
  $\lim_k|(\mathcal{H}_{\bar\varepsilon,k}^{\mathrm{succ}})^*| =
  |\mathcal{H}_{\bar\varepsilon}^*\cap A | \geq (1-2\kappa)|\mathcal{H}_{\bar\varepsilon}^*\cap A |
  >0$. Since also
  $\lim_k|(\mathcal{H}_{\bar\varepsilon,k}^{\mathrm{succ}})^*\cap A| =
  |\mathcal{H}_{\bar\varepsilon}^*\cap A |$ we obtain
  $\lim_k(|\mathcal{H}_{\bar\varepsilon}^*|/|\mathcal{H}_{\bar\varepsilon}^*\cap A |) = 1$ which
  yields the desired conclusion.
\end{proof}

\section{Strategical reproducibility of the Haar system}
\label{sec:strat-repr-haar}

We establish that the Haar system is strategically reproducible in the following classical Banach
spaces:
\begin{enumerate}[(i)]
\item The multi-parameter tensor product Haar system is strategically reproducible in the reflexive
  mixed norm Lebesgue spaces $L^{(p_1,\ldots,p_d)}$, $1 < p_i < \infty$, $1 \leq i \leq d$,
  $d\in\mathbb{N}$, in the sense of Definition~\ref{D:1.1}.
\item The one-parameter Haar system in $H^1$ is strategically reproducible according to
  Definition~\ref{D:1.2}.
\item The two-parameter tensor product Haar system is strategically reproducible in the
  two-parameter mixed norm Hardy spaces $H^p(H^q)$, $1\leq p,q <\infty$ in the sense of
  Definition~\ref{D:1.2}.
\end{enumerate}

\subsection{The Haar system in multi-parameter Lebesgue spaces}\label{sec:haar-system-multi}

Here we show that $(h_{\bar I})$ is strategically reproducible in $L^{(p_1,\ldots,p_d)}$,
$1 < p_i < \infty$, $1 \leq i \leq d$, $d\in\mathbb{N}$, in the sense of Definition~\ref{D:1.1}.  We
exploit the fact that $(h_{\bar I})$ is an unconditional basis for $L^{(p_1,\ldots,p_d)}$, and that
$(L^{(p_1,\ldots,p_d)})^* = L^{(q_1,\ldots,q_d)}$, where $\frac{1}{p_i} + \frac{1}{q_i} = 1$,
$1\leq i \leq d$.

\begin{thm}
  Let $d\in\N$, $1<p_1,p_2, \ldots, p_d<\infty $ and put $\bar p=(p_1,p_2,\ldots,p_d)$. Then
  $(h_{\bar I})$ with an appropriate linear order is strategically reproducible in $L^{\bar p}_0$.
\end{thm}

\begin{proof}
  We linearly order $\mathcal{R}_d$ into
  $(\bar I^{(k)})=(I^{(k)}_1,I^{(k)}_2, \ldots, I^{(k)}_d)_{k=1}^\infty$ in a manner which is
  compatible with ``$\subset$'', \ie, we assume that if for $m,k\kin\N$, we have
  $I^{(m)}_i\subset I^{(k)}_i$, for $i=1,2,\ldots,d$, then $m\ge k$. We also linearly order the Haar
  basis of $L^{\bar p}_0$ into $(h_{\bar I^{(k)}})_{k=1}^\infty$.  For any
  $\bar I\in \mathcal{R}_d$, $n(\bar I)$ denotes the number $n\in\N$ so that $\bar I= \bar I^{(n)}$.

  Then a winning strategy for player (II) will look as follows:

  Assume he has chosen $b_1,b_2,\ldots b_l$, and $b^*_1,b^*_2,\ldots b_l^*$ in $L^{\bar p}_0$ and
  $(L^{\bar p}_0)^*$. Assume that $b_j$, $1\le j\le l$ is of the following form:
  \begin{equation*}
    b_j= \sum_{
      \begin{matrix}
        \scriptstyle \bar J=(J_1,\ldots J_d)\in \cD_{k(j,1)}\times \cD_{k(j,2)}\times \ldots \times \cD_{k(j,d)}\\
        \scriptstyle J_1\times J_2\times \ldots \times J_d\subset I^{(j)}_1\times
        I^{(j)}_2\times\ldots I^{(j)}_d
      \end{matrix}
    } h_{\bar J}
  \end{equation*}
  with $k(j,i)\le k(j',i)$, if $j\le j'$, and $k(j,i)< k(j',i)$ if $I^{(j)}_i\supsetneq I^{(j')}_i$,
  for $i=1,2\ldots d$. We also assume that for all $j=1,2\ldots l$ we have $n(\bar J)>n_j$ for all
  $\bar J\in \cD_{k(j,2)}\times \ldots \times \cD_{k(j,d)}$, where $n_j$ was the $j$-th move of
  player (I) and, moreover, we assume that
  \begin{equation*}
    b^*_j=\Big(\prod_{s=1}^d |I_s^{(j)}|\Big)^{-1} b_j.
  \end{equation*}
  Thus $(b^*_j)_{j=1}^l$ is biorthogonal to $(b_j)_{j=1}^l$ and $|b_j|=|h_{\bar I^{(j)}}|$, for
  $j=1,2\ldots l$ which means that with respect to $\trin\cdot\trin_{\bar p}$ and, using
  \eqref{E:2.6.1} in Proposition \ref{P:2.6}, $(b_j)_{j=1}^l$ and $(b^*_j)_{j=1}^l$ are
  isometrically equivalent to $(h_{\bar I_j})_{j=1}^l$ and $(h^*_{\bar I_j})_{j=1}^l$, respectively.

  Assuming now the $(l+1)$st move of player (I) is $n_{l+1}$, player (II) can proceed as follows.
  For $j=1,2\ldots d$, he chooses $k(l+1,j)\ge \max_{m\le l} k(m,j)$ so that $k(l+1,j)>k(m,j)$ if
  $I^{(l+1)}_j \subsetneq I^{(m)}_j$ and so that for all $\bar J\in \prod_{j=1}^d \cD_{k(l+1,j)}$,
  we have $n(\bar J)\ge n_{l+1}$.
\end{proof}

\subsection{The Haar system in $H^1$}
\label{sec:h1}

Here, we use the Gamlen-Gaudet construction~\cite{gamlen:gaudet:1973} (see
also~\cite{1987:mueller,mueller:2005}) to show that the one-parameter Haar system in $H^1$ is
strategically reproducible according to Definition~\ref{D:1.2}.

For convenience, we introduce the following notation: We define $e_I = h_I/|I|$, $I\in\mathcal{D}$,
thus $(e_I)_{I\in\mathcal{D}}$ forms a $1$-unconditional normalized basis for $H^1$. Note that
$e_I^* = h_I\in (H^1)^*$, $I\in\mathcal{D}$.  Finally, we will identify a dyadic interval
$I\in\mathcal{D}$ with $\mathcal{O}(I)$, e.g.
\begin{equation*}
  \mathcal{E}_k \leftrightarrow \mathcal{E}_I,
  \qquad x_k \leftrightarrow x_I,
  \qquad x_k^* \leftrightarrow x_I^*,
  \qquad I\in\mathcal{D},\ \mathcal{O}(I) = k.
\end{equation*}
Recall that the linear order $\mathcal{O}$ was introduced in Section~\ref{sec:multi-parameter-haar}.

\begin{thm}\label{thm:H1-strat-rep}
  The normalized Haar basis is strategically reproducible in $H^1$.
\end{thm}

\begin{proof}
  Here, we show that $(e_I)_{I\in\mathcal{D}}$ is $\sqrt{2}$-strategically reproducible in $H^1$.

  We start the game with turn $1$.  In step $1$, player (I) chooses $\eta_{[0,1)}>0$,
  $W_{[0,1)}\in\mathrm{cof}(H^1)$, and $G_{[0,1)}\in\mathrm{cof}_{w^*}((H^1)^*)$.  In step $2$,
  player (II) selects one of the sets $\mathcal{E}_{[0,1)}^{(j)}= \mathcal{D}_j$, $j\in\mathbb{N}$.
  Put
  \begin{align*}
    d_{[0,1)}^{(j)}
    &= \sum_{K\in \mathcal{E}_{[0,1)}^{(j)}} |K| e_K,\\
    d_{[0,1)}^{*(j)}
    &= \sum_{K\in \mathcal{E}_{[0,1)}^{(j)}} e_K^*,
  \end{align*}
  and note that $(d_{[0,1)}^{(j)})_{j=1}^\infty$ converges to $0$ in the weak topology of $H^1$ and
  the sequence $(d_{[0,1)}^{*(j)})_{j=1}^\infty$ converges in the w$^*$ topology in $(H^1)^*$.
  Hence, there exists an index $j_0$ such that
  $\dist_{H^1}(d_{[0,1)}^{(j_0)}, W_{[0,1)}) < \eta_{[0,1)}$ and
  $\dist_{(H^1)^*}(d_{[0,1)}^{*(j_0)}, G_{[0,1)}) < \eta_{[0,1)}$.  Player (II) concludes step 2 by
  choosing
  \begin{equation*}
    \mathcal{E}_{[0,1)}
    = \mathcal{E}_{[0,1)}^{(j_0)}
  \end{equation*}
  and
  \begin{equation*}
    \lambda_K^{[0,1)}
    = |K|
    \quad\text{and}\quad
    \mu_K^{[0,1)}
    = 1,
    \qquad K\in \mathcal{E}_{[0,1)}.
  \end{equation*}
  In step $3$, player $(I)$ chooses
  $(\varepsilon_{K}^{([0,1))})_{K\in \mathcal{E}_{[0,1)}}\in\{-1,1\}^{\mathcal{E}_{[0,1)}}$.

  Assume that the game has already been played for $k=\mathcal{O}(I)-1$ turns.  We will now play out
  turn $k+1 = \mathcal{O}(I)$.  In step $1$, player (I) chooses $\eta_I > 0$,
  $W_I\in\mathrm{cof}(H^1)$, and $G_I\in\mathrm{cof}_{w^*}((H^1)^*)$.  In step $2$, it is player
  II's choice to select the finite sets $\mathcal{E}_I\subset\mathcal{D}$.  We will now describe
  this procedure.  Note that since $W_I\in\mathrm{cof}(H^1)$ and
  $G_I\in\mathrm{cof}_{w^*}((H^1)^*)$, there exist $f_j\in H^1$, $g_j\in (H^1)^*$,
  $1\leq j \leq N_I$, such that $W_I = \{g_1,\ldots,g_{N_I}\}_\perp$ and
  $G_I = \{f_1,\ldots,f_{N_I}\}^\perp$.  Let $\widetilde I$ denote the dyadic predecessor of $I$,
  i.e. $\widetilde I$ is the unique dyadic interval that satisfies $\widetilde I \supsetneq I$ and
  $|\widetilde I| = 2|I|$.  Note that $\mathcal{O}(\widetilde I)\leq k$; specifically,
  $\mathcal{E}_{\widetilde I}$ has already been defined.  Put
  \begin{equation*}
    X_I =
    \begin{cases}
      \{ K^\ell : K\in \mathcal{E}_{\widetilde I}\}, & \text{if $I$ is the left successor of $\widetilde I$},\\
      \{ K^r : K\in \mathcal{E}_{\widetilde I}\}, & \text{if $I$ is the right successor of
        $\widetilde I$},
    \end{cases}
  \end{equation*}
  where $K^\ell$ denotes the left successor of $K$ and $K^r$ denotes the right successor of $K$.  We
  note that by induction $|X_I| = |I|$.
  \begin{align*}
    d_I^{(j)}
    &= \sum_{\substack{K\in \mathcal{D}_j\\K\subset X_I}} |K| e_K,
    \qquad j\in\mathbb{N},\\
    d_I^{*(j)}
    &= \sum_{\substack{K\in \mathcal{D}_j\\K\subset X_I}} e_K^*,
    \qquad j\in\mathbb{N}.
  \end{align*}
  Since the sequence $(d_I^{(j)})_{j=1}^\infty$ converges to $0$ in the weak topology of $H^1$ and
  $(d_I^{^*(j)})_{j=1}^\infty$ converge to $0$ in the w$^*$ topology of $(H^1)^*$, there exists an
  index $j_0$ such that
  \begin{equation*}
    \supp(d_I^{(j_0)}) = X_I,
    \qquad
    2^{-j_0} < \min\Bigl\{ |K| : K\in \bigcup_{i=1}^k \mathcal{E}_i\Bigr\},
  \end{equation*}
  as well as
  \begin{equation*}
    \dist_{H^1}(d_I^{(j_0)}, W_I) < \eta_I,
    \qquad
    \dist_{(H^1)^*}(d_I^{*(j_0)}, G_I) < \eta_I.
  \end{equation*}
  Player (II) concludes step $2$ by choosing the collection
  \begin{equation*}
    \mathcal{E}_I = \{K\in\mathcal{D}_{j_0} : K\subset X_I\},
  \end{equation*}
  the numbers
  \begin{equation*}
    \lambda_K^{(I)} = |K|/|I|,\quad K\in \mathcal{E}_I,
    \qquad\text{and}\qquad
    \mu_K^{(I)} = 1,\quad K\in \mathcal{E}_I,
  \end{equation*}
  and defining
  \begin{equation*}
    x_I = d_I^{(j_0)}
    \qquad\text{and}\qquad
    x_I^* = d_I^{*(j_0)}.
  \end{equation*}
  Clearly, since $|X_I| = |\mathcal{E}_I^*| = |I|$, we have
  $\sum_{K\in \mathcal{E}_I} \lambda_K^{(I)}\mu_K^{(I)} = 1$.  We conclude turn
  $k+1 = \mathcal{O}(I)$ with player (I) choosing $(\varepsilon_K^{(I)})_{K\in \mathcal{E}_I}$ in
  $\{-1,1\}^{\mathcal{E}_I}$ in step 3.

  We claim that this defines a winning strategy for player (II).  In fact, $(x_I)_{I\in\mathcal{D}}$
  and $(x_I)_{I\in\mathcal{D}}^*$ were constructed using the Gamlen-Gaudet construction, for which
  it was verified in~\cite[Theorem~0]{1987:mueller} that $(x_I)_I$ is equivalent to
  $(e_I)_I\in\mathcal{D}$ in $H^1$ and $(x_I^*)_I$ is equivalent to $(e_I^*)_I\in\mathcal{D}$ in
  $(H^1)^*$ .  Moreover, we note that in the text above we already verified
  $\dist_{H^1}(x_I, W_I) < \eta_I$ and $\dist_{(H^1)^*}(x_I^*, G_I) < \eta_I$, $I\in\mathcal{D}$.
\end{proof}

\subsection{The Haar system in $H^p(H^q)$, $1\leq p,q <\infty$}
\label{sec:haar-system-h1}

In~\cite{laustsen:lechner:mueller:2015} it was shown that the two-parameter tensor product Haar
system in $H^p(H^q)$ has the factorization property.  We use the techniques introduced
in~\cite{laustsen:lechner:mueller:2015}, to show that, moreover, the two-parameter tensor product
Haar system is strategically reproducible $H^p(H^q)$.  Hence, by Theorem~\ref{thm:strat-rep:1} we
recover the main result in~\cite{laustsen:lechner:mueller:2015}.  Finally, we remark that by
exploiting Theorem~\ref{thm:sums:strat:rep:1} (see Section \ref{sec:uncond-sums-spac} below) we
obtain a simpler construction than the one used in~\cite{laustsen:lechner:mueller:2015}.

In the proof below, we will use following notation: The $H^p(H^q)$-normalized bi-parameter Haar
system $(e_I\otimes f_J)_{I,J}$ is given by $e_I\otimes f_J = h_I\otimes h_J/(|I|^{1/p}|J|^{1/q})$,
and its bi-orthogonal functionals $((e_I\otimes f_J)^*)_{I,J}$ are given by
$(e_I\otimes f_J)^* = e_I^*\otimes f_J^* = h_I\otimes h_J/(|I|^{1/p'}|J|^{1/q'})$, where
$1/p + 1/p' = 1$ and $1/q + 1/q' = 1$, with the usual convention that $1/\infty = 0$.

\begin{thm}\label{thm:HpHq:strat-rep}
  The normalized bi-parameter Haar system is strategically reproducible in the mixed norm Hardy
  spaces $H^p(H^q)$, $1 \leq p ,q \leq \infty$.
\end{thm}

\begin{myproof}
  Define the subspaces $V_1$, $V_2$ of $H^p(H^q)$ by
  \begin{equation*}
    V_1 = \overline{\spa\{ h_{I\times J} : |I| < |J|\}}^{H^p(H^q)}
    \quad\text{and}\quad
    V_2 = \overline{\spa\{ h_{I\times J} : |I| \geq |J|\}}^{H^p(H^q)},
  \end{equation*}
  and note that $V_1\oplus V_2 = H^p(H^q)$.  We will now show that the bi-parameter Haar system is
  strategically reproducible in $V_1$ and in $V_2$, separately.

  \begin{proofcase}[Case~$V_1$]
    Here, we will show that the bi-parameter Haar system is strategically reproducible in $V_1$.
    Player (I) opens the game by selecting $\eta_{[0,1/2)\times [0,1)} > 0$,
    $W_{[0,1/2)\times [0,1)}\in \cof(H^p(H^q))$ and
    $G_{[0,1/2)\times [0,1)}\in \cof_{w^*}((H^p(H^q))^*)$.  In step 2, player (II) will select one
    of the following collections of dyadic rectangles $\cE_{[0,1/2)\times[0,1)}^{(j)}$,
    $j\in\mathbb{N}$, given by
    \begin{equation*}
      \cE_{[0,1/2)\times[0,1)}^{(j)}
      = \{ K\times [0,1) : K\in\mathcal{D}_j,\ K\subset [0,1/2)\},
      \qquad j\in\mathbb{N} . 
    \end{equation*}
    To this end, define
    \begin{align*}
      d_{[0,1/2)\times[0,1)}^{(j)}
      &= \sum_{K\times L\in \cE_{[0,1/2)\times[0,1)}^{(j)}}
        \frac{|K|^{1/p} |L|^{1/q}}{|[0,1/2)|^{1/p} |[0,1)|^{1/q}} e_K\otimes f_L,
        \qquad j\in\mathbb{N},\\
      d_{[0,1/2)\times[0,1)}^{^*(j)}
      &= \sum_{K\times L\in \cE_{[0,1/2)\times[0,1)}^{(j)}}
        \frac{|K|^{1/p'} |L|^{1/q'}}{|[0,1/2)|^{1/p'}|[0,1)|^{1/q'}}
        e_K^*\otimes f_L^*,
        \qquad j\in\mathbb{N},
    \end{align*}
    and note that $K\times L\in \cE_{[0,1/2)\times[0,1)}^{(j)}$ implies $L = [0,1)$.  By
    Lemma~\ref{lem:w-w*-convergence} the sequence $(d_{[0,1/2)\times[0,1)}^{(j)})_{j=1}^\infty$ is a
    null sequence in the weak topology of $H^p(H^q)$ and
    $(d_{[0,1/2)\times[0,1)}^{*(j)})_{j=1}^\infty$ is a null sequence in the w$^*$ topology of
    $H^p(H^q)^*$, we can find $j_0$, such that
    \begin{align*}
      \dist_{H^p(H^q)}\big(d_{[0,1/2)\times[0,1)}^{(j_0)}, W_{[0,1/2)\times [0,1)}\big)
      &< \eta_{[0,1/2)\times [0,1)},\\
      \dist_{(H^p(H^q))^*}\big(d_{[0,1/2)\times[0,1)}^{*(j_0)}, G_{[0,1/2)\times [0,1)}\big)
      &< \eta_{[0,1/2)\times [0,1)}.
    \end{align*}
    Player II completes step 2 by selecting
    \begin{equation*}
      \cE_{[0,1/2)\times[0,1)}
      = \cE_{[0,1/2)\times[0,1)}^{(j_0)}
    \end{equation*}
    and the numbers
    \begin{align*}
      \lambda_{K\times L}^{[0,1/2)\times[0,1)}
      &= \frac{|K|^{1/p} |L|^{1/q}}{|[0,1/2)|^{1/p} |[0,1)|^{1/q}},
      &&K\times L \in \cE_{[0,1/2)\times[0,1)},\\
      \mu_{K\times L}^{[0,1/2)\times[0,1)}
      &= \frac{|K|^{1/p'} |L|^{1/q'}}{|[0,1/2)|^{1/p'} |[0,1)|^{1/q'}},
      &&K\times L \in \cE_{[0,1/2)\times[0,1)}.
    \end{align*}
    Finally, in step 3, player (I) chooses
    \begin{equation*}
      (\varepsilon_{K\times L}^{[0,1/2)\times[0,1)})_{K\times L \in \cE_{[0,1/2)\times[0,1)}}
      \in \{-1,+1\}^{\cE_{[0,1/2)\times[0,1)}},
    \end{equation*}
    completing turn 1.

    Assume that the turns $1,\ldots,k$ of the game have been played out.  We will now describe turn
    $k+1$.  Select $I, J\in\mathcal{D}$ such that $|I| < |J|$ and $I\times J$ is the
    $(k+1)^{\text{st}}$ rectangle of the set $\{ K\times L : |K| < |L|\}$ in the order $\drless$.
    Player I starts off the turn by choosing $\eta_{I\times J} > 0$,
    $W_{I\times J}\in \cof(H^p(H^q))$ and $G_{I\times J}\in \cof_{w^*}((H^p(H^q))^*)$.  In step 2,
    player (II) will select one of the collections $\cE_{I\times J}^{(j)}$, $j\in\mathbb{N}$ of
    dyadic rectangles, which we will now describe in detail.  We will distinguish between the
    following two cases: $J = [0,1)$ and $J\neq [0,1)$.

    If $J=[0,1)$, we note that $I\neq [0,1)$, hence, $\cE_{\widetilde I\times [0,1)}$ has already
    been defined.  Also note that $[0,|I|)\times [0,1)\drlesseq I\times [0,1)$, and choose the
    integer $\kappa(I\times [0,1))$ such that
    \begin{equation*}
      2^{-\kappa(I\times [0,1))}
      < \min\Bigl\{ |K| : K\times L\in\bigcup_{i=1}^k \cE_i \Bigr\}.
    \end{equation*}
    For all $j\in\mathbb{N}$, we define
    \begin{equation*}
      \cE_{I\times [0,1)}^{(j)}
      =
      \begin{cases}
        \{ K^+\times [0,1) : K\in\mathcal{D}_j,\ K\subset [0,1)\subset \cE_{\widetilde I\times
          [0,1)}\},
        &\text{if $I = (\widetilde I)^\ell$}\\
        \{ K^-\times [0,1) : K\in\mathcal{D}_j,\ K\subset [0,1)\subset \cE_{\widetilde I\times
          [0,1)}\}, &\text{if $I = (\widetilde I)^r$},
      \end{cases}
    \end{equation*}
    where $J^\ell$ denotes the left successor and $J^r$ the right successor of a dyadic interval
    $J\in\mathcal{D}$.  Define the functions
    \begin{align*}
      d_{I\times [0,1)}^{(j)}
      &= \sum_{K\times L\in \cE_{I\times [0,1)}^{(j)}}
        \frac{|K|^{1/p}|L|^{1/q}}{|I|^{1/p}|[0,1)|^{1/q}}
        e_K\otimes f_L,
        \qquad j\in\mathbb{N},\\
      d_{I\times [0,1)}^{*(j)}
      &= \sum_{K\times L\in \cE_{I\times [0,1)}^{(j)}}
        \frac{|K|^{1/p'}|L|^{1/q'}}{|I|^{1/p'}|[0,1)|^{1/q'}}
        e_K^*\otimes f_L^*,
        \qquad j\in\mathbb{N},
    \end{align*}
    and note that by Lemma~\ref{lem:w-w*-convergence} $(d_{I\times [0,1)}^{(j)})_{j=1}^\infty$
    converges to $0$ in the weak topology of $H^p(H^q)$ and
    $(d_{I\times [0,1)}^{*(j)})_{j=1}^\infty$ converges to $0$ in the w$^*$ topology of
    $(H^p(H^q))^*$.  Hence, there exists and index $j_0 > \kappa(I\times [0,1))$ such that
    \begin{align*}
      \dist_{H^p(H^q)}(d_{I\times[0,1)}^{(j_0)}, W_{I\times [0,1)})
      &< \eta_{I\times [0,1)},\\
      \dist_{(H^p(H^q))^*}(d_{I\times[0,1)}^{*(j_0)}, G_{I\times [0,1)})
      &< \eta_{I\times [0,1)}.
    \end{align*}
    Player II then completes step 2 by putting
    \begin{equation*}
      \cE_{I\times[0,1)}
      = \cE_{I\times[0,1)}^{(j_0)}
    \end{equation*}
    and selecting the numbers
    \begin{align*}
      \lambda_{K\times L}^{I\times[0,1)}
      &= |K|^{1/p}|L|^{1/q}/(|I|^{1/p}|[0,1)|^{1/q}),\\
      \mu_{K\times L}^{I\times[0,1)}
      &= |K|^{1/p'}|L|^{1/q'}/(|I|^{1/p'}|[0,1)|^{1/q'}).
    \end{align*}

    If $J\neq [0,1)$, then $\cE_{I\times [0,1)}$ has already been defined.  The collection
    $\cE_{I\times J}$ will be chosen as one of the following collections $\cE_{I\times J}^{(j)}$,
    $j\in\mathbb{N}$, which are given by
    \begin{equation*}
      \cE_{I\times J}^{(j)}
      = \{ K\times L : K\in\mathcal{D}_j,\ K\times L\subset \cE_{I\times [0,1)} \},
      \qquad j\in\mathbb{N}.
    \end{equation*}
    To this end, define the functions
    \begin{align*}
      d_{I\times J}^{(j)}
      &= \sum_{K\times L\in \cE_{I\times J}^{(j)}}
        \frac{|K|^{1/p}|L|^{1/q}}{|I|^{1/p}|J|^{1/q}}
        e_k\otimes f_L,
        \qquad j\in\mathbb{N},\\
      d_{I\times J}^{*(j)}
      &= \sum_{K\times L\in \cE_{I\times J}^{(j)}}
        \frac{|K|^{1/p'}|L|^{1/q'}}{|I|^{1/p'}|J|^{1/q'}}
        e_k^*\otimes f_L^*,
        \qquad j\in\mathbb{N},
    \end{align*}
    and note that by Lemma~\ref{lem:w-w*-convergence} $(d_{I\times J}^{(j)})_{j=1}^\infty$ converges
    to $0$ in the weak topology of $H^p(H^q)$ and that $(d_{I\times J}^{(j)})_{j=1}^\infty$
    converges to $0$ in the w$^*$ topology of $(H^p(H^q))^*$.  Consequently, we can find an index
    $j_0$ such that
    \begin{equation*}
      2^{-j_0}
      < \min\Bigl\{ |K| : K\times L \in \bigcup_{i=1}^k \cE_{i}\Bigr\},
    \end{equation*}
    and
    \begin{align*}
      \dist_{H^p(H^q)}(d_{I\times J}^{(j_0)}, W_{I\times  J})
      &< \eta_{I\times  J},\text{ and }
        \dist_{(H^p(H^q))^*}(d_{I\times J}^{*(j_0)}, G_{I\times  J})
        < \eta_{I\times  J}.
    \end{align*}
    Player II completes step 2 by selecting
    \begin{equation*}
      \cE_{I\times J}
      = \cE_{I\times J}^{(j_0)}
    \end{equation*}
    and the numbers
    \begin{equation*}
      \lambda_{K\times L}^{I\times J}
      = \frac{|K|^{1/p}|L|^{1/q}}{|I|^{1/p}|J|^{1/q}}
      \qquad\text{and}\qquad
      \mu_{K\times L}^{I\times J}
      = \frac{|K|^{1/p'}|L|^{1/q'}}{|I|^{1/p'}|J|^{1/q'}},
    \end{equation*}
    for all $K\times L \in \cE_{I\times J}$.
    
    Finally, in both cases ($J=[0,1)$ and $J\neq [0,1)$) player (I) completes step 3 (and thereby
    turn $(k+1)$) player (I) chooses
    \begin{equation*}
      (\varepsilon_{K\times L}^{I\times J})_{K\times L \in \cE_{I\times J}}
      \in \{-1,+1\}^{\cE_{I\times J}}.
    \end{equation*}
  \end{proofcase}

  \begin{proofcase}[Case~$V_2$]
    Follows from a completely parallel argument to Case~$V_1$.\qedhere
  \end{proofcase}

\end{myproof}

\section{The Haar system in $L^1$ is strategically reproducible}
\label{sec:haar-system-l1_0-1}
In this section we consider the space $L^1$ of all absolutely integrable functions on $[0,1]$
instead of $L^1_0$. If we additionally define $h_\emptyset = \chi_{[0,1)}$ and
$\mathcal{D}^+ = \mathcal{D}\cup\{\emptyset\}$. We call then $(h_I)_{I\in\mathcal{D}^+}$ is a
monotone Schauder basis of $L^1$, if ordered lexicographically (i.e., $\emptyset$ is the minimum of
$\mathcal{D}^+$ and the rest of the order is inherited from the lexicographical order of
$\mathcal{}D$). The reason we consider $L^1$ is that we can prove an isometric statement in this
setting and it is unclear whether this is possible for the space $L_0^1$. The purpose is to prove
that the normalized Haar system of $L^1$ is strategically reproducible, and thus has the
factorization property.  The main difficulty is proving the following statement.

\begin{thm}
  \label{L1 s-r}
  The normalized Haar system of $L^1$ is 1-strategically reproducible.
\end{thm}

The proof of the above will be presented at it in its own Subsection~\ref{sec:proof-theorem-refl1}.
We will also need the following statement.

\begin{prop}
  \label{L1 d-f}
  The normalized Haar system of $L^1$ has the $1/\delta$-diagonal factorization property.
\end{prop}

We will give the proof of Proposition~\ref{L1 d-f} in Subsection~\ref{L1 d-f subsection}.  For the
time being, we use the above two results to prove the following Corollary.
\begin{cor}
  \label{L1 fact prop}
  The normalized Haar system of $L^1$ has the $1/\delta$-factorization property.
\end{cor}

\begin{proof}
  By Theorem \ref{L1 s-r} the normalized Haar system is 1-strategically reproducible and by
  Proposition \ref{L1 d-f} it has the $1/\delta$-factorization property. Since the Haar system is
  monotone, Theorem \ref{fun and games factor through one's identity} yields that it has the
  $1/\delta$-factorization property.
\end{proof}

The following is a direct consequence of Corollary \ref{L1 fact prop}.
\begin{cor}\label{cor:L^1(L^1)}
  The normalized bi-parameter Haar system of $L^1(L^1)$ (which is isometrically isomorphic to
  $L^1([0,1]^2)$) has the $1/\delta$-factorization property.
\end{cor}

\begin{proof}
  Let $T:L^1(L^1)\to L^1(L^1)$ be a bounded linear operator so that for all $I,J\in\mathcal{D}$ we
  have
  \begin{equation*}
    \Big|\int\int (h_I\otimes h_J)T\Big(\frac{1}{|I||J|}h_I\otimes h_J\Big)dxdy\Big| \geq
    \frac{1}{\delta}.
  \end{equation*} The space $L^1(L^1)$ is the projective tensor product of $L^1$ with
  itself. This means that if we consider any two bounded linear operators $R,S$ on $L^1$ then
  there is a (unique) bounded linear operator $R\otimes S:L^1(L^1)\to L^1(L^1)$ satisying
  $(R\otimes S)(f\otimes g) =(Rf)\otimes(Sg)$ and $\|R\otimes S\| = \|R\|\|S\|$.

  Consider the canonical projection $P_{[0,1)}$ from $L^1$ onto the linear span of $h_{\emptyset}$,
  which has norm one, and also consider the identity $I_{L^1}$ on $L^1$. Take the map
  $P = P_{[0,1)}\otimes I_{L^1}:L^1(L^1)\to L^1(L^1)$, which satisfies $\|P\| = 1$.  Its image is
  the space $Y = [h_{\emptyset}\otimes h_I]_{I\in\mathcal{D}}$, which is naturally isometric to
  $L^1$ via the map $h_{\emptyset}\otimes h_I\mapsto h_I$. It follows that the map $P\circ T:Y\to Y$
  may be identified with a map on $L^1$ with diagonal bounded below by $\delta$. This means, by
  Theorem \ref{L1 d-f}, that for $\varepsilon>0$ there are $B:Y\to L^1$, $A:L^1\to Y$ so that
  $B(PT)A = I_{L^1}$ and $\|A\|\|B\| \leq(1+\varepsilon)/\delta$. Now, since $L^1$ and $L^1(L^1)$
  are isometrically isomorphic we may take an onto isometry $Q:L^1\to L^1(L^1)$ and set
  $\tilde A = AQ^{-1}:L^1(L^1)\to L^1(L^1)$, $\tilde B = QBP:L^1(L^1)\to L^1(L^1)$. It follows that
  $\|\tilde A\|\|\tilde B\| \leq(1+\varepsilon)/\delta$ and $\tilde BT\tilde A = I$.
\end{proof}

\begin{remark}
  One can use \cite[Theorem 4.2]{EnfloStarbird1979} to give a, relatively, short proof of the
  following. There is $C\geq 1$ so that if $T:L^1\to L^1$ is a bounded linear operator with diagonal
  bounded below by $\delta$ then there are $A,B:L^1\to L^1$ with $BTA = I$ and
  $\|B\|\|A\| \leq C\|T\|/\delta^2$. There are two differences with Theorem \ref{L1 d-f}. The first
  one is the power appearing on $\delta$. The more noticeable one is that in Theorem \ref{L1 d-f}
  the factorization does not depend on $\|T\|$. This seems to be the case for all known spaces with
  the factorization property.
\end{remark}

\begin{remark}
  We point out that Theorem \ref{L1 s-r}, Proposition \ref{L1 d-f}, and Corollary \ref{L1 fact prop}
  are true for the normalized Haar system of $L^p[0,1]$, $1\leq p<\infty$ as well. The proof of
  Theorem \ref{L1 s-r} requires only minor modifications and in certain cases it is simpler due to
  reflexivity and unconditionality. The proof of Proposition \ref{L1 d-f} is different and one has
  to use the details of the proof of Theorem \ref{L1 s-r}. The factorization property of these
  spaces has been known since \cite{Andrew1979}, however existing proofs did not give a sharp
  factorization estimate, in particular it was not known whether the Haar system of $L^p[0,1]$,
  $1\leq p<\infty$, has the $C/\delta$-factorization property for a uniform constant $C\geq 1$.
\end{remark}

\subsection{The proof of Proposition \ref{L1 d-f}}
\label{L1 d-f subsection}
We now turn to the proof of Proposition \ref{L1 d-f}.  We divide the argument into several steps
formulated in Lemma~\ref{1-comp} and~\ref{stabiliziation of diagonal}, below.  Lemma~\ref{1-comp} is
likely to be know. We present a short proof for sake of completeness and convenience of the reader.

\begin{lem}
  \label{1-comp}
  Let $I_0\in\mathcal{D}$ and $Y_{I_0} = [(h_I)_{I\subset I_0}]$. Then there are a subspace $Z$ of
  $Y_{I_0}$ that isometrically isomorphic to $L^1$ and a norm one linear projection $P:L^1\to Z$.
\end{lem}

\begin{proof}
  Set $m_0 = \inf I_0$, $M_0 = \sup I_0$ and let $Z$ be the subspace consisting of all absolutely
  integrable functions $f$ that have support in $I_0$ and satisfy the condition
  \begin{equation*}
    f(x) = - f\Big(\frac{m_0 + M_0}{2} + x\Big) \text{ a.e. in }I_0^+.
  \end{equation*} It follows that $W$ is a
  subspace of $Y_{I_0}$. An onto isometry $T:Z\to L^1$ is given by
  \begin{equation*}
    Tf(x) =  |I_0|f\Big(m_0 + \frac{|I_0|}{2}x\Big).
  \end{equation*}
  The desired projection $P:L^1\to Z$ is defined as follows. For every $f\in L^1$ let
  $f_1 = f|_{I_0^+}$, $f_2 = f|_{I_0^-}$, let $g_1$ be the function with support $I_0^+$ so that
  $g_1(x) = f(x + |I_0|/2)$, for $x\in I_0^+$, and let $g_2$ be the function with support $I_0^-$ so
  that $g_2(x) = f(x - |I_0|/2)$, for $x\in I_0^-$. One can check that
  $\|f_1 + f_2\|_{L^1} = \|g_1 + g_2\|_{L^1} = \|f\|_{L^1}$. It is also not hard to see that
  $(f_1 + f_2) - (g_1+g_2) = (f_1-g_1) + (f_2 - g_2)$ is in $W$. The final step is to observe that
  $Pf = (1/2)[(f_1 + f_2) - (g_1+g_2)]$ is a norm one projection onto $Z$.
\end{proof}

It is common to call diagonal operators on the Haar system Haar multipliers. Loosely following
\cite{SemenovUksusov2012} we use the following notation.
\begin{notation}~
  \begin{itemize}
  \item[(i)] A chain of $\mathcal{D}^+$ is a sequence of intervals $\mathcal{C} = (I_n)_n$ so that
    $I_1\supsetneq I_2\supsetneq\cdots$.
  \item[(ii)] Given a Haar multiplier $D$ with entries $(c_I)_{I\in\mathcal{D}^+}$ we define the
    quantity $\|D\|_W = \sup\sum_{n=1}^\infty|c_{I_n} - c_{I_{n+1}}|$, where the supremum is taken
    over all chains $\mathcal{C} = (I_n)_n$.
  \item[(iii)] Given a Haar multiplier $D$ with entries $(c_I)_{I\in\mathcal{D}^+}$ we define the
    quantity $\|D\|_\infty = \sup_{I\in\mathcal{D}^+}|c_I|$.
  \end{itemize}
\end{notation}

\begin{lem}
  \label{stabiliziation of diagonal}
  Let $D:L^1\to L^1$ be a bounded Haar multiplier with entries $(c_I)_{I\in\mathcal{D}}$. Then for
  every $\varepsilon>0$ there exists $I_0\in\mathcal{D}$ so that if
  $Y_{I_0} = [(h_{I})_{I\subset I_0}]$ then
  \begin{equation*}
    \|(D-c_{I_0} I)|_{Y_{I_0}}\| \leq \varepsilon.
  \end{equation*}
\end{lem}

\begin{proof}
  According to \cite[Theorem 3]{SemenovUksusov2012} (see end of page 312), for any Haar multiplier
  $D:L^1\to L^1$ we have
  \begin{equation}
    \label{diagonal norm}
    \frac{1}{4}\|D\|_W\leq \|D\| \leq \|D\|_W + 3\|D\|_\infty.
  \end{equation}
  We use the above to choose $I_0\in\mathcal{D}$ with the property that for any chain
  $\mathcal{C} = (I_n)$ with $I_n\subset I_0$ we have
  $\sum_n|c_{I_n} - c_{I_{n+1}}| \leq \varepsilon/4$. If such an $I_0$ would not exist then it would
  easily follow that $\|D\|_W = \infty$ which by \eqref{diagonal norm} contradicts the boundedness
  of $D$.

  Let $P_{I_0}$ define the canonical projection onto $Y_{I_0}$, given by
  $P(\sum_{I\in\mathcal{D}}a_Ih_I) = \sum_{I\subset I_0}a_Ih_I$, which has norm at most
  two. Clearly, $P_{I_0}$ is a Haar multiplier. Next, define the Haar multiplier $S$ with entries
  $(\tilde c_I)_{I\in\mathcal{D}^+}$, whith $\tilde c_I = c_I$ is $I\subset I_0$ and
  $\tilde c_I = 0$ otherwise. It is easy to see that
  $\|S - c_{I_0}P_{I_0}\|_W = \|S\|_W\leq \varepsilon/4$ and that
  $\|S - c_{I_0}P_{I_0}\|_\infty \leq \varepsilon/4$. Therefore by \eqref{diagonal norm} we deduce
  $\|S - P_{I_0}\| \leq \varepsilon$. Since $D|_{Y_{I_0}} = S|_{Y_{I_0}}$ and and
  $I|_{Y_{I_0}} = P_{I_0}|_{Y_{I_0}}$ we finally conclude
  $\|(D - c_{I_0}I)|_{Y_{I_0}}\| = \|(S - c_{I_0}P_{I_0})|_{Y_{I_0}}\| \leq \|S -
  c_{I_0}P_{I_0}\|\leq \varepsilon$.
\end{proof}

We are ready to conclude this subsection with a proof of Proposition \ref{L1 d-f}.
\begin{proof}[Proof of Proposition \ref{L1 d-f}]
  Let $D$ be a bounded Haar multiplier on $L^1$ satisfying
  $\inf_{I\in\mathcal{D}^+}|c_I| \geq \delta >0$ and fix $\varepsilon>0$. We will show that the
  identity $(1+\varepsilon)/\delta$-factors through $D$. Use Lemma \ref{stabiliziation of diagonal}
  to find $I_0\in\mathcal{D}$ so that
  $\|(D - c_{I_0}I)|_{Y_{I_0}}\| \leq \delta\varepsilon/(1+\varepsilon)$.

  By Lemma \ref{1-comp} there are a subspace $Z$ of $Y_{I_0}$, an onto isometry $A:L^1\to Z$, and a
  norm one projection $P:L^1\to Z$. Define $B = (1/c_{I_0})A^{-1}P$, which is well defined on $L^1$
  and $\|B\| = |1/c_{I_0}|\leq 1/\delta$. As the image of $A$ is $Z$ it easily follows that
  $B(c_{I_0}I)A = I$. We calculate
  \begin{equation*}
    \|BDA - I\| = \|B(D - c_{I_0}I)A\| = \|B(D - c_{I_0}I)|_ZA\| \leq
    \|B\|\|A\|\frac{\delta\varepsilon}{1+\varepsilon} \leq \frac{\varepsilon}{1+\varepsilon}
  \end{equation*} and
  hence the operator $BDA$ is invertible with $\|(BDA)^{-1}\| \leq 1+\varepsilon$. Define
  $\tilde B = (BDA)^{-1}B$. Observe that $\tilde BDA = I$ and
  $\|\tilde B\|\|A\| \leq (1+\varepsilon)/\delta$, i.e., the identity almost $1/\delta$-factors
  through $D$.
\end{proof}

\subsection{The proof of Theorem \ref{L1 s-r}}
\label{sec:proof-theorem-refl1}
The following Lemma~\ref{middle stepone} is a well known result, which goes back to
Gamlen-Gaudet~\cite{gamlen:gaudet:1973}.  For more details, we also refer
to~\cite[page~176~ff.]{mueller:2005}. It describes the situation player (II) is striving to achieve
in order to win the game of strategic reproducibility.  Recall that for
$\mathcal{A}\subset\mathcal{D}$ we set $\mathcal{A}^* = \cup\mathcal{A}$.
\begin{lem}
  \label{middle stepone}
  Let $\kappa\in(0,1)$. Then, there exists a sequence of positive real numbers $(\delta_n)_n$ so
  that the following holds.  Let $(\mathcal{H}_I)_{I\in\mathcal{D}^+}$ be a collection of non-empty
  finite subsets of $\mathcal{D}$, so that for each $I\in\mathcal{D}^+$ the collection
  $\mathcal{H}_I$ consists of pairwise disjoint intervals, and for each $I\in\mathcal{D}^+$ let
  $\bar \varepsilon_I = (\varepsilon^I_L)_{L\in\mathcal{H}_I}\in\{-1,1\}^{\mathcal{H}_I}$. Define
  for each $I\in\mathcal{D}$ the function $b_I = \sum_{L\in\mathcal{H}_I}\varepsilon^I_Lh_L$ and
  $b_\emptyset = |\sum_{L\in\mathcal{H}_\emptyset}\varepsilon^I_Lh_L|$. Assume that the following
  are satisfied.
  \begin{itemize}

  \item[(a)] For all $I,J\in\mathcal{D}$ with $I\cap J = \emptyset$ we have
    $\mathcal{H}_I^*\cap\mathcal{H}_J^* = \emptyset$.

  \item[(b)] For all $I\in\mathcal{D}$ we have $\mathrm{supp}(b_{I^+})\subset [b_I = 1]$ and
    $\mathrm{supp}(b_{I^-})\subset [b_I = -1]$.

  \item[(c)] For all $n\in\mathbb{N}$ and $I\in\mathcal{D}_n$ if $I = J^+$ or $I = J^-$ we have
    \begin{equation*}
      (1-\delta_n)\frac{|\mathcal{H}_J^*|}{2}\leq
      |\mathcal{H}_I^*|\leq(1+\delta_n)\frac{|\mathcal{H}_J^*|}{2}.
    \end{equation*}
  \item[(d)] $\mathcal{H}^*_{[0,1)}\subset\mathcal{H}^*_\emptyset$ and
    $|\mathcal{H}^*_{[0,1)}| \geq (1-\delta_1)|\mathcal{H}^*_\emptyset|$.

  \end{itemize}
  Then, if $\lambda = |\mathcal{H}_{\emptyset}^*|$, the sequences $(h_I/|I|)_{I\in\mathcal{D}^+}$
  and $(b_I/\lambda|I|)_{I\in\mathcal{D}^+}$, when they are both viewed as sequences in $L^1$, are
  $(1+\kappa)$-impartially equivalent. Furthermore, the sequences $(b_I)_{I\in\mathcal{D}^+}$ and
  $(h_I)_{I\in\mathcal{D}^+}$, when they are both viewed as a sequence in $L^\infty$, are
  isometrically equivalent.
\end{lem}

The following Lemma allows player (II) to make the appropriate choice of vectors.
\begin{lem}
  \label{approx subsp}
  Let $H$ be in $\mathrm{cof}(L^1)$, $G$ be in $\mathrm{cof}_{w^*}(L^\infty)$, and $\kappa >
  0$. Then there exists $n_0\in\mathbb{N}$ so that for every $f$ in the linear span of
  $(h_I)_{I\in\mathcal{D}\setminus \mathcal{D}^{n_0}}$ with $\|f\|_\infty\leq 1$ we have
  \begin{itemize}

  \item[(i)] $\mathrm{dist}(f,H)<\kappa$, if $f$ is viewed as an element of $L^1$ and

  \item[(ii)] $\mathrm{dist}(f,H)<\kappa$, if $f$ is viewed as an element of $L^\infty$.

  \end{itemize}

\end{lem}

\begin{proof}
  We first show (i). Recall that there are $g_1,\ldots,g_N\in L^\infty$ so that
  $H = \cap_{j=1}^N\mathrm{ker}g_j$. It follows, from the Hahn Banach theorem, that there is
  $\delta>0$ so that for every $f\in L^1$ with $|\int g_j(x)f(x) dx| < \delta$ for $1\leq j\leq N$
  we have $\mathrm{dist}(f,H)<\kappa$. If we assume that the conclusion is false, i.e. the desired
  $k_0$ does not exist, there is a sequnce $(f_k)_k$ with $\|f_k\|_\infty\leq1$ and
  $f_k\in\mathrm{span}\{h_I: I\in\mathcal{D}\setminus\mathcal{D}^k\}$, so that
  $\mathrm{dist}(f_k,H)\geq \kappa$ for all $k\in\mathbb{N}$. As this sequence is uniformly
  integrable it has a subsequence $(f_{k_i})_i$ that converges weakly to an
  $f\in\cap_k\overline{\mathrm{span}}\{h_I: I\in\mathcal{D}\setminus\mathcal{D}^k\} = \{0\}$. Thus,
  $\lim_i|\int g_j(x)f_{k_i}(x)dx = 0$, for $1\leq j\leq N$, i.e.
  $\lim_k\mathrm{dist}(f_{k_i},H) = 0$, which is a contradiction.

  The second statement follows from a similar argument. Use that there are $g_1,\ldots,g_N$ in $L^1$
  so that $G = \{f_1,\ldots,f_N\}^\perp$ and that for any sequence $(g_k)_k$ with
  $\|g_k\|_\infty\leq 1$ and $g_k\in\mathrm{span}\{h_I: I\in\mathcal{D}\setminus\mathcal{D}^k\}$,
  for all $k\in\mathbb{N}$, we have that $(g_k)_k$ converges to zero in the $w^*$-topology.
\end{proof}

We refer to Section~\ref{sec:coll-dyad-interv} for the notation employed systematically in the
proof, below.
\begin{proof}[Proof of Theorem \ref{L1 s-r}]
  Enumerate $\mathcal{D}^+$ as $(I_k)_{k\in\mathbb{N}}$ according to lexicographical order. We will
  describe the winning strategy of player (II) in a game of $\mathrm{Rep}_{(L^1_0,(h_I))}(1,\eta)$,
  for fixed $\eta > 0$. Before the game starts player (I) picks a partition
  $\mathbb{N} = N_1\cup N_2$, which corresponds to a partition
  $\mathcal{D}^+ = \mathcal{A}_1\cup\mathcal{A}_2$. Before proceeding with the game, we claim that
  $[0,1) = \limsup\mathcal({A}_1)\cup \lim\sup\mathcal({A}_2)$. Indeed, if $x\in[0,1)$ then for
  every $I\in\mathcal{D}$ with $x\in I$ we have $I\in\mathcal{A}_1$ or $I\in\mathcal{A}_2$. That is,
  $x\in I$ for infinitely many $I\in\mathcal{A}_1$ or for infinitely many $I\in\mathcal{A}_2$. In
  the first case $x\in\limsup\mathcal({A}_1)$ and in the second case
  $x\in\limsup\mathcal({A}_2)$. We proceed by assuming without loss of generality that
  $|\limsup(\mathcal{A}_1)|\geq 1/2$. By Lemma \ref{assume finite} there is
  ${\mathcal{A}}\subset\mathcal{A}_1$ so that $\mathscr{G}_n({\mathcal{A}})$ is finite for all
  $n\in\mathbb{N}$ and $|\limsup({\mathcal{A}})| =\lambda \geq 2/3$. Henceforth, when we use
  Notation \ref{level partition} it shall be with respect to the collection $\mathcal{A}$. This
  collection $\mathcal{A}$ corresponds to some $N\subset N_1$. Each round $k$ corresponds to an
  $I\in\mathcal{D}^+$ via the lexicographical identification
  $\mathcal{D}^+\leftrightarrow\N$. Player (I) first chooses $\eta_k>0$, $W_k\in\mathrm{cof}(L^1)$
  and $G_k\in\mathrm{cof}_{w^*}(L^\infty)$. Then player (II) has the right to choose a subset $E_k$
  of either $N_1$ or $N_2$. He or she will always choose $E_k\subset N$.  This $E_k$ corresponds to
  a finite $\mathcal{H}_I\subset\mathcal{A}$. We shall describe the choice in detail further bellow,
  but let us say for the time being that there is $m_k\in\N$, with $m_k>m_{k-1}$ if $k>1$, so that
  $\mathcal{H}_I\subset \mathcal{A}_{m_k}$ and
  $\mathcal{H}_I^*\subset\mathcal{H}_\emptyset^*$. Next, player (I) chooses signs
  $(\varepsilon^k_i)_{i\in E_k}$ which we relabel as $(\varepsilon_L^I)_{L\in\mathcal{H}_I}$. We
  shall put $b_I = \sum_{L\in\mathcal{H}_I}\varepsilon_L^Ih_I$ and $\tilde b_I = b_\emptyset b_I$
  (pointwise). If $k>1$ then by the fact that $m_k>1$ and
  $\mathcal{H}_I^*\subset\mathcal{H}_\emptyset^*$ we have that
  $\tilde b_I = \sum_{L\in\mathcal{H}_I}\tilde\varepsilon_L^Ih_I$, for a choice of signs
  $(\tilde \varepsilon_L^I)_{I\in\mathcal{H}_I}$ that does not necessarily coincide with
  $(\varepsilon_L^I)_{L\in\mathcal{H}_I}$. Of some importance is also the sequence of positive real
  numbers $(\delta_n)_n$ provided by Lemma \ref{middle stepone} if we take $\kappa = \eta$.

  We can now describe how player (II) makes a choice in each round $k$. Let $I\in\mathcal{D}^+$
  correspond to $k$ in the lexicographical enumeration. Then, either $I = \emptyset$ (if $k=1$),
  $I = [0,1)$ (if $k=2$), or there is $2\leq k'<k$ so that if $J = I_{k'}$ then $I = J^+$ or
  $I = J^-$. The round starts by player (I) picking $\eta_k>0$.  Player (II) will pick
  $E_k\subset N$ that corresponds to an $\mathcal{H}_I\subset \mathcal{A}$ which is chosen as
  follows:
  \begin{itemize}

  \item[(i)] There is $m_k\in\mathbb{N}$, with $m_{k}>m_{k-1}$ if $k>1$, so that $\mathcal{H}_I$ is
    of one of the following forms
    \begin{itemize}

    \item[(ia)] $\mathcal{H}_I = \mathscr{G}_{m_k}(\mathcal{A})$, if $I = \emptyset$ or $I = [0,1)$
      (\ie, when $k=1$ or $k=2$),

    \item[(ib)] $\mathcal{H}_I = (\mathcal{H}_J)_{\bar\varepsilon_J,m_k}^\mathrm{succ}$, if
      $I = J^+ = I_{k'}$ and $\bar\varepsilon_J = (\tilde\varepsilon_L^J)_{J\in\mathcal{H}_J}$
      coming from $\tilde b_J$.

    \item[(ic)] $\mathcal{H}_I = (\mathcal{H}_J)_{\text{-}\bar\varepsilon_J,m_k}^\mathrm{succ}$, if
      $I = J^- = I_{k'}$ and $\bar\varepsilon_J = (\tilde\varepsilon_L^J)_{J\in\mathcal{H}_J}$
      coming from $\tilde b_J$.

    \end{itemize}

  \item[(ii)] If $A = \limsup(\mathcal{A})$ then if
    $|\mathcal{H}_I^*\cap A| > (1-\delta_{n+1}/2)|\mathcal{H}_I^*|$, where $I\in\mathcal{D}_n$ and
    $(\delta_i)_i$ is the sequence mentioned above, provided by Lemma \ref{middle stepone}. If
    $I=\emptyset$ replace $n$ with $0$.

  \item[(iii)] For every $f\in\mathrm{span}(h_L)_{L\in\mathcal{D}\setminus\mathcal{D}^{m_k-1}}$ with
    $\|f\|_\infty \leq 1$ we have
    \begin{equation*}
      \mathrm{dist}_{L^1}(f,W_k) \leq \eta_k|I|\text{ and }\mathrm{dist}_{L^\infty}(f,G_k) \leq
      \eta_k.
    \end{equation*}

  \end{itemize}
  Having chosen such an $\mathcal{H}_I$ player (II) picks scalars
  $(\lambda_L^I)_{L\in\mathcal{H}_I}$, $(\mu_L^I)_{L\in\mathcal{H}_I}$ by taking
  $\lambda_L^I = |L|/(|\mathcal{H}_\emptyset||I|)$ and $\mu_L^I = 1$, for all $L\in\mathcal{H}_I$.

  We must show that player (II) can pick $\mathcal{H}_I$ satisfying (i), (ii), and (iii) as well as
  that
  \begin{equation}
    \label{close to one}
    1-\eta \leq \sum_{L\in\mathcal{H}_I}\lambda_L^I\mu_L^I \leq 1+\eta.
  \end{equation}

  Note that if $m_k$ is sufficiently large then by Lemma \ref{approx subsp} condition (iii) is
  satisfied. We can focus on showing that we can pick $m_k$, as large as desired, so that
  $\mathcal{H}_I$ is of one of the forms in (i) and so that (ii) is satisfied. If $I=\emptyset$ and
  $k=1$ then $A = \cap_m\mathscr{G}_m(\mathcal{A})^*$ then
  $|A| = \lim_n|\mathscr{G}_{m}(\mathcal{A})^*|$ which easily yields that we can pick $m_1$ as large
  as we wish so that
  $|\mathscr{G}_{m_1}(\mathcal{A})^*\cap A| = |A| > (1-
  \delta_2/2)|\mathscr{G}_{m_1}(\mathcal{A})^*|$. If $I = [0,1)$ we act similarly. If $I = J^+$ with
  $I\in\mathcal{D}_n$ then by assumption player (II) has picked $\mathcal{H}_J\subset\mathcal{A}$
  with $|\mathcal{H}^*_J\cap A| > (1-\delta_n/2)|\mathcal{H}_J^*|$. Let also
  $\bar\varepsilon_J = (\tilde\varepsilon_L^J)_{J\in\mathcal{H}_J}$ denote the signs coming from
  $\tilde b_J$. By Lemma \ref{eventual choices} there exists $m_0$ so that for any $m\geq m_0$ we
  have
  \begin{equation*}
    |((\mathcal{H}_J)_{\bar\varepsilon_J,m}^\mathrm{succ})^*\cap A| >
    (1-\delta_{n+1}/2)|((\mathcal{H}_J)_{\bar\varepsilon_J,m}^\mathrm{succ})^*|.
  \end{equation*}
  Thus, if we pick $m_k$ sufficiently large we may set
  $\mathcal{H}_I = (\mathcal{H}_J)_{\bar\varepsilon_J,m}^\mathrm{succ}$ and (i) and (ii) are
  satisfied. If $I = J^-$ the argument is the same.

  The proof of \eqref{close to one} requires an inductive argument. We will show this simultaneously
  with proving that player (II) has forced the desired winning conditions. Define for
  $I\in\mathcal{D}^+$ the functions
  \begin{equation*}
    x_I = \sum_{L\in\mathcal{H}_I}\varepsilon_L^I\frac{|L|}{|\mathcal{H}_\emptyset| |I|}h_I\in L^1\text{ and
    }x_I^* = \sum_{L\in\mathcal{H}_I}\varepsilon_L^Ih_I\in L^\infty.
  \end{equation*}
  Out next goal is to show that $(x_I)_{I\in\mathcal{D}^+}$ is $(1+\eta)$-impartially equivalent to
  $(h_I/|I|)_{I\in\mathcal{D}^+}$ in $L^1$ and $(x_I^*)_{I\in\mathcal{D}}$ is isometrically
  equivalent to $(h_I)_{I\in\mathcal{D}}$ in $L^\infty$.  We will first show that
  $(\tilde b_I)_{I\in\mathcal{D}^+}$, where
  $\tilde b_\emptyset = b_\emptyset^2 = |\sum_{L\in\mathcal{H}_\emptyset}\varepsilon_L^\emptyset
  h_L|$ and for each $I\in\mathcal{D}$
  $\tilde b_I = \sum_{L\in\mathcal{H}_I}\tilde \varepsilon_L^Ih_I$, satisfies the assumptions of
  Lemma \ref{middle stepone} and use that to reach the desired goal.

  Assumption (a) follows easily from (i) and assumption (b) follows easily from (ib) and (ic). Let
  us now show that (c) is satisfied and let $n\in\mathbb{N}$, $I\in\mathcal{D}_n$, with $I = I_k$,
  so that $I = J^+$ or $I = J^-$ for some $J\in\mathcal{D}_{n-1}$. We shall assume that $I = J^+$ as
  the other case has the same proof. By (iii), applied to $J$, we have that
  $|\mathcal{H}_J^*\cap A| > (1-\delta_n/2)|\mathcal{H}_J^*|$. This, by the first statement of Lemma
  \ref{eventual choices} yields that
  \begin{equation*}
    |(\mathcal{H}_J)_{\bar\varepsilon_J}^*\cap A| > (1-\delta_n)\frac{|\mathcal{H}_J^*|}{2}.
  \end{equation*}
  By the definition of $(\mathcal{H}_J)^\mathrm{succ}_{\bar\varepsilon_J,m_k}$ it follows that
  $(\mathcal{H}_J)_{\bar\varepsilon_J}^*\cap A = (\mathcal{H}_J)_{\bar\varepsilon_J,m_k}^*\cap A =
  \mathcal{H}_I^*\cap A$. We calculate,
  \begin{equation*}
    \frac{|\mathcal{H}_J^*|}{2} = |(\mathcal{H}_J)_{\bar\varepsilon_J}^*| \geq |\mathcal{H}_I^*| \geq
    |\mathcal{H}_I^*\cap A| >(1-\delta_n)\frac{|\mathcal{H}_J^*|}{2},
  \end{equation*}
  i.e., (c) holds. Assumption (d) is easier to show.
  
  Now that we know that the assumptions of Lemma \ref{middle stepone} are satisfied we conclude that
  the sequences
  \[\tilde x_I = \sum_{L\in\mathcal{H}_I}\tilde \varepsilon_L^I\frac{|L|}{|\mathcal{H}_\emptyset|
      |I|}h_I = \in L^1\text{ and }\tilde x_I^* = \sum_{L\in\mathcal{H}_I}\tilde
    \varepsilon_L^Ih_I\in L^\infty. \] are $(1+\eta)$-impartially equivalent to
  $(h_I/|I|)_{I\in\mathcal{D}^+}$ in $L^1$ and isometrically equivalent to $(h_I)_{I\in\mathcal{D}}$
  in $L^\infty$ respectivelly. We next observe that for $I\in \mathcal{D}^+$ we have
  $\tilde x_I = b_\emptyset x_I$ and $\tilde x_I^* = b_\emptyset x_I^*$. But $|b_\emptyset(t)|$ is
  one whenever $t\in\mathcal{H}_\emptyset^*\supset\supp(x_I)=\supp(x^*_{I})$ and thus
  $(x_I)_{I\in\mathcal{D}^+}$ is isometrically equivalent to $(\tilde x_I)_{I\in \mathcal{D}^+}$ and
  $(x^*_I)_{I\in\mathcal{D}^+}$ is isometrically equivalent to
  $(\tilde x^*_I)_{I\in \mathcal{D}^+}$. This means that we reached our goal.

  Next, we need to observe that if $I = I_k$ then $\mathrm{dist}_{L^1}(x_{I},W_k) < \eta_k$ as well
  as $\mathrm{dist}_{L^\infty}(x^*_{I},G_k) < \eta_k$. Both of these inequalities are an immediate
  consequence of (iii) and the fact that
  $x_k,x_k^*\in\mathrm{span}(h_L)_{L\in\mathscr{G}_{n_k}(\mathcal{A})}\subset
  \mathrm{span}(h_L)_{L\in\mathcal{D}\setminus\mathcal{D}^{m_k-1}}$ and $\|x_I\|_{L^\infty} = |I|$,
  $\|x_I^*\|_{L^\infty} = 1$.

  It only remains to prove \eqref{close to one}, which follows easily from the fact that $(x_I)_I$
  is $(1+\eta)$-impartially equivalent to $(h_I/|I|)_I$. Indeed,
  \begin{equation*}
    \sum_{L\in\mathcal{H}_I}\lambda_L^I\mu_L^I = \frac{1}{\lambda |I|}\sum_{L\in\mathcal{H}_I}|L| =
    \|x_I\|_{L^1}
  \end{equation*}
  and $\|x_I\|_{L^1}$ is between
  \begin{equation*}
    \sqrt{(1-\eta)}(\|h_I\|_{L^1}/|I|) >1-\eta
    \qquad\text{and}\qquad
    \sqrt{(1+\eta)}(\|h_I\|_{L^1}/|I|) < 1+\eta.
  \end{equation*}
  The proof is complete.
\end{proof}

\section{Unconditional sums of spaces with strategical reproducible bases}
\label{sec:uncond-sums-spac}

In this section we determine that the strategical reproducibility is inherited by unconditional
sums.

For a Banach spaces $X$ with a 1-unconditional basis $(e_n)_n$ and a sequence of Banach spaces
$(Y_n)_n$ we denote by $Z = (\sum Y_n)_X$ the Banach space of all sequences $z = (y_n)_n$ with
$y_n\in Y_n$ for all $n\in\N$ and the quantity
\begin{equation*}
  \left\|z\right\| = \left\|\sum_{n=1}^\infty\|y_n\|_{Y_n}e_n\right\|_X
\end{equation*}
is well defined. For each $k\in\N$ let $P_k:Z\to Y_k$ denote the map given by $P_k(y_n)_n =
y_k$. The space $Y_k$ can be naturally isometrically identified with a subspace of Z, namely the one
consisting of all sequences which have all coordinates, except the $k$'th one, equal to zero. Thus,
with this identification, $P_k$ is a norm one projection.

\begin{remark}
  \label{operator direct sum}
  If $A_n:Y_n\to Y_n$, $n\in\N$, are bounded linear operators and $\sup_n\|A_n\| = C< \infty$ then
  by 1-unconditionality the map $A:Z\to Z$ with $A(z) = \sum_nA_nP_n(z)$ is bounded with $\|A\| =C$.
\end{remark}

\begin{remark}
  \label{appropriate enumeration}
  If there exists a common $\lambda\geq 1$ such that each $Y_n$ has a Schauder basis $(e_i^{(n)})_i$
  whose basis constant is bounded by $\lambda$ then there is an enumeration $(\tilde e_i)$ of
  $((e_i^{(n)})_i)_n$ that is Schauder basic whose basis constant at most $\lambda$. In fact, this
  is satisfied by any enumeration $(\tilde e_i)$ with the property that whenever $i<j$ if for some
  $n\in\N$ we have $e_{i}^{(n)} = \tilde e_{k_i^{(n)}}$, $e_j^{(n)} = \tilde e_{k_j^{(n)}}$ then
  $k_i^{(n)} < k_j^{(n)}$.
\end{remark}

\begin{lem}
  \label{diagonal restricted isomorphism goes through}
  Let $X$ be a Banach space with a 1-unconditional basis $(e_n)_n$, $(Y_n)_n$ be a sequence of
  Banach spaces, and $Z = (\sum Y_n)_X$. Assume that there are common $\lambda\geq 1$ and
  $K:(0,+\infty)\to \mathbb{R}$ so that each $Y_n$ has a Schauder basis $(e_i^{(n)})_i$ whose
  constant is at most $\lambda$ that has the $K(\delta)$-diagonal factorization property. Then the
  sequence $((e_i^{(n)})_i)_n$ is a Schauder basis (using the linear order defined in
  Remark~\ref{appropriate enumeration}) whose basis constant is at most $\lambda$ and it has the
  $K(\delta)$-diagonal factorization property.
\end{lem}

\begin{proof}
  Let $D:X\to X$ be a diagonal operator with respect to $((e_i^{(n)})_i)_n$ so that
  $\inf_{i,n}|e_i^{(n)*}D(e_i^{(n)})|>\delta$. If follows that for each $n\in\N$ the map $D$
  restricted on $Y_n$ is a diagonal operator $D_n$ so that
  $\inf_{i,n}\big|e_i^{(n)*}\big(D_n(e_i^{(n)})\big)\big|>\delta$. For $\kappa>0$, by assumption,
  there exist $B_n,A_n:Y_n\to Y_n$ with $\|A_n\|\|B_n\| \leq K(\delta)+\kappa$ and $B_nD_nA_n$ is
  the identity map on $Y_n$. By scaling, we may assume that
  $\max\{\|A_n\|,\|B_n\|\} \leq \sqrt{K(\delta)+\kappa}$ and hence the maps $A,B:Z\to Z$, with
  $A(z) = \sum_nA_nP_n(z)$ and $B(z) = \sum_nB_nP_n(z)$ are well defined with
  $\|A\|\|B\| \leq K(\delta)+\kappa$. It is easily verified that $I = BDA$.
\end{proof}

\begin{lem}
  \label{restricting gives a good enough estimate}
  Let $X$ be a Banach space with a 1-unconditional basis $(e_n)_n$, $(Y_n)_n$ be a sequence of
  Banach spaces, and $Z = (\sum Y_n)_X$. Fix $n\in\N$ and let $A$, $B$ be finite subsets of $Z$ and
  $Z^*$ respectively. Define
  \begin{align*}
    A_n &= \{P_n(x): x\in A\},\;
          B_n = \{P^*_n(x^*): x^*\in B\},\\
    G &= A^\perp,\;
        H = \bigcap_{x^*\in B}\ker(x^*),\;
        G_n = A_n^\perp,\;
        W_n = \bigcap_{x^*\in B_n}\ker(x^*).
  \end{align*}
  Then, for every $x\in Y_n$ and $x^*\in Y_n^*$ we have
  $\mathrm{dist}(x,H) \leq \mathrm{dist}(x,W_n)$ and
  $\mathrm{dist}(x^*,G) \leq \mathrm{dist}(x^*,G_n)$.
\end{lem}

\begin{proof}
  For every $y\in W_n$ it follows that $P_n(y)\in H$. Hence
  $\|x - y\| \geq \|P_n(x - y)\| = \|x - P_n(y)\|\geq \mathrm{dist}(x,H)$ and so
  $\mathrm{dist}(x,W_n)\geq\mathrm{dist}(x,H)$. Similarly, for $f\in G_n$ we have $P_n^*f\in G$ and
  we conclude in the same manner that $\mathrm{dist}(x^*,G) \leq \mathrm{dist}(x^*,G_n)$.
\end{proof}

\begin{prop}
  \label{reproducible goes through}
  Let $X$ be a Banach space with a 1-unconditional basis $(e_n)_n$, $(Y_n)_n$ be a sequence of
  Banach spaces, and $Z = (\sum Y_n)_X$. Assume that there are common $\lambda\geq 1$ and $C\geq 1$
  so that each $Y_n$ has a $C$-strategically reproducible Schauder basis $(e_i^{(n)})_i$ whose basis
  constant is at most $\lambda$. Then the sequence $((e_i^{(n)})_i)_n$ enumerated as $(\tilde e_n)$
  according to Remark \ref{appropriate enumeration} is a $C$-strategically reproducible Schauder
  basis whose basis constant is at most $\lambda$.
\end{prop}

\begin{proof}
  For each $n\in\N$ let $M_n$ be the infinite subset of $\N$ so that
  $(\tilde e_i)_{i\in M_n} = (e_i^{(n)})_i$. Here, we will now describe the winning strategy of
  player (II) in a game $\mathrm{Rep}_{(X,(\tilde e_i))}(C,\eta)$. Let player (I) pick a partition
  $\N = N_1\cup N_2$. Note that for each $n\in\N$ $M_n = M_n^1\cup M_n^2$, where
  $M_n^1 = N_1\cap M_n$, $M_n^2 = N_2\cap M_n$. The $m$'th round is played out as follows. Player
  (I) selects $\eta_m>0$ as well as $W_n\in \mathrm{cof}(Z)$, $G_n\in\mathrm{cof}_{w^*}(Z^*)$. Then,
  there are finite subsets $A_m$ and $B_m$ of $Z$ and $Z^*$ respectively so that $G_m = A_m^\perp$
  and $W_m = (B_m)_\perp$. If $m\in M_n$, for some $n\in\N$, and $m$ is the $k$'th element of $M_n$
  then set $A_k^{(n)} = \{P_nx: x\in A\}$ and $B_k^{(n)} = \{P_n^*f: f\in B\}$. Then set
  $G_k^{(n)} = (A_k^{(n)})^\perp$, $W_k^{(n)} = (B_k^{(n)})^\perp$. Let player (II) treat this round
  as the $k$'th round of a game $\mathrm{Rep}_{(Y_n,(e^{(n)}_i))}(C,\eta)$ and follow a winning
  strategy. In the end, for each $n\in\N$, player (II) has chosen $(x^{(n)}_k)_k$ in $Y_n$ and
  $(x^{(n)*}_k)_{k}$ in $Y_n^*$ so that
  \begin{itemize}
  \item[(i)] the sequences $(x^{(n)}_k)_k$ and $(e_k^{(n)})_k$ are impartially
    $(C+\eta)$-equivalent,
  \item[(ii)] the sequences $(x^{(n)*}_k)_k$ and $(e_k^{(n)*})_k$ are impartially
    $(C+\eta)$-equivalent,
  \item[(iii)] for all $k\in\N$, if the $k$'th element of $M_n$ is $m$, we have
    $\mathrm{dist}(x_k^{(n)}, W_k^{n}) < \eta_m$,
  \item[(iv)] for all $k\in\N$, if the $k$'th element of $M_n$ is $m$, we have
    $\mathrm{dist}(x_k^{(n)*}, G_k^{n}) < \eta_m$.
  \end{itemize}
  If we relabel $(x^{(n)}_k)_k$ as $(\tilde x_m)_{m\in M_n}$ and stitch them all together to a
  sequence $(\tilde x_m)_{m\in\N}$ then it easily follows that this sequence is impartially
  $(C+\eta)$-equivalent to $(\tilde e_m)_m$. Also by Lemma \ref{restricting gives a good enough
    estimate} we have $\mathrm{dist}(x_m, W_m) < \eta_m$, for all $m\in\N$. Similarly, relabel
  $(x^{(n)*}_k)_k$ as $(\tilde x^*_m)_{m\in M_n}$ and take $(\tilde x^*_m)_m$, which is
  $(C+\eta)$-impartially equivalent to $(\tilde e_m^*)_m$. Also,
  $\mathrm{dist}(x^*_m, G_m) < \eta_m$, for all $m\in\N$. In other words, player two has emerged
  victorious.
\end{proof}

\begin{thm}\label{thm:sums:strat:rep:1}
  \label{good diagonal and reproducible}
  Let $X$ be a Banach space with a 1-unconditional basis $(e_n)_n$, $(Y_n)_n$ be a sequence of
  Banach spaces, and $Z = (\sum Y_n)_X$. Assume that there are common $\lambda\geq 1$, $C\geq 1$,
  and $K:(0,+\infty)\to\mathbb{R}$ so that each $Y_n$ has a Schauder basis $(e_i^{(n)})_i$ that
  satisfies the following:
  \begin{itemize}
  \item[(i)] its basis constant is at most $\lambda$,
  \item[(ii)] it has the $K(\delta)$-diagonal factorization property and
  \item[(iii)] it is $C$-strategically reproducible in $Y_n$.
  \end{itemize}
  Then the sequence $((e_i^{(n)})_i)_n$ enumerated as $(\tilde e_n)$ (using the linear order defined
  in Remark~\ref{appropriate enumeration}) has the $\lambda C^{2}K(\delta)$-factorization property.
\end{thm}

\begin{proof}
  By Lemma \ref{diagonal restricted isomorphism goes through} the basis of $Z$ is $\lambda$-basic
  and it has the $K(\delta)$-diagonal factorization property. By Proposition \ref{reproducible goes
    through} the basis of $Z$ is $C$-strategically reproducible. We finish off the proof by using
  Proposition \ref{reproducible goes through}.
\end{proof}

\begin{remark}
  A consequence of the above theorem is that if one takes $X = \ell_p$, $1\leq p<\infty$ of
  $X = c_0$ and a sequence of spaces $(Y_n)_n$, with each $Y_n$ being some $L^p$ or $\ell_p$,
  $1\leq p<\infty$, or $c_0$ then $Z = (\sum Y_n)_X$ has a 1-strategically reproducible basis. Of
  course, $X$ can be any space of an unconditional basis and this produces an interesting
  example. It was proved in \cite{laustsen:lechner:mueller:2015} that Gowers' space $X$ with an
  unconditional basis from \cite{gowers:1994} does not satisfy the factorization property, however,
  for this $X$ and any sequence $(Y_n)_n$ as before the space $Z = (\sum Y_n)_X$ does satisfy it.
\end{remark}

\section{Final comments and open problems}
\label{sec:final-comments-open}

Capon~\cite{capon:1982:2} showed that the bi-parameter Haar system in $L^p(L^q)$,
$1 < p,q < \infty$, has the factorization property.  With refined techniques, Capon's result was
extended to $H^1(H^1)$ by~\cite{mueller:1994} and then later in~\cite{laustsen:lechner:mueller:2015}
to $H^p(H^1)$ and $H^1(H^p)$, $1 < p < \infty$.  In Section~\ref{sec:strat-repr-haar}, we gave a
different proof of their results in, by writing $H^p(H^q)$ as a complemented sum of two spaces,
solving the problem in each of the components separately, and then using the fact that strategically
reproducibility is inherited by complemented sums.  This begs the following question.

\begin{prob}
  If $X$ and $Y$ are Banach spaces with bases that have the factorization property, does the union
  of those two bases (in the right order) have the factorization property in the complemented sum of
  $X$ and $Y$?
\end{prob}

As we remarked after Theorem \ref{thm:strat-rep:1} the uniform factorization property from
Definition \ref{rip} is formally stronger than the factorization property from Definition
\ref{factorization definitions} (iii).
\begin{prob}
  Is there a Banach space with a basis that satisfies the factorization property and fails the
  uniform factorization property?
\end{prob}

In Corollary~\ref{cor:L^1(L^1)} we showed that the bi-parameter Haar system has the factorization
property.  Nevertheless, we do not know the answer to the following problem.
\begin{prob}
  Is the normalized bi-parameter Haar system of $L^1(L^1)$ strategically reproducible?
\end{prob}

The unit vector basis of Tsirelson space $T$, i.e. the space constructed by Figiel and Johnson
in~\cite{FigielJohnson1974}, which is the dual of Tsirelson's original space, is \emph{not}
strategically reproducible.  This follows from the following two facts.  On the one hand, every
block bases $(x_i)$ is equivalent to a subsequence $(e_{n_i})$ with $n_i\in\supp(x_i)$,
$i\in\mathbb{N}$.  Secondly, if the subsequence is an Ackerman sequence (i.e. is increasing fast
enough), then $[e_{n_i}]$ is not isomorphic to $T$.  Thus, if player (I) chooses an Ackerman
sequence in the game described in Definition~\ref{D:1.2}, he wins.  This leads to the next problem.
\begin{prob}
  Does the unit vector basis in $T$ have the factorization property?
\end{prob}

Among all the bi-parametric Lebesgue and Hardy spaces, $L^p(L^1)$ and $L^1(L^p)$, $1 < p < \infty$,
seem to resist our approaches.
\begin{prob}
  Is the bi-parameter Haar system strategically reproducible or does it at least have the
  factorization property in $L^p(L^1)$ and $L^1(L^p)$, $1 < p < \infty$?

  More generally, if $X$ has a basis which is strategically reproducible, or has the factorization
  property, does the tensor product of that basis with the Haar system in $L^p(X)$,
  $1\leq p < \infty$ have the same property?
\end{prob}
In this context it is worth noting that Capon~\cite{capon:1983} proved $L^p(X)$, $1 \leq p < \infty$
is primary, if $X$ has a symmetric basis.

\end{document}